\documentclass[letterpaper,11pt,UKenglish,cleveref, autoref, thm-restate]{article}
\usepackage{fullpage}
\usepackage{amsmath,amsfonts,amsbsy,amsgen,amscd,mathrsfs,amssymb,amsthm,comment}
\usepackage{enumerate,mathtools}
\usepackage{authblk}
\usepackage{bm}

\usepackage{amsmath, amsthm, amssymb, mathtools, graphicx, hyperref, color, braket, wasysym, geometry, bbm, color, xfrac, multirow}
\usepackage{tikz}
\usetikzlibrary{shadings,matrix,fit,decorations.pathreplacing,calc}

\usepackage{enumitem}
\setlist{parsep=0pt,listparindent=\parindent}
\usepackage{ wasysym }
\title{Quantum and classical low-degree learning via a dimension-free Remez inequality}
\author[1]{Ohad Klein}
\author[2]{Joseph Slote}
\author[3]{Alexander Volberg}
\author[4]{Haonan Zhang}
\affil[1]{\small School of Computer Science and Engineering, Hebrew University\protect\\ {\footnotesize\texttt{ohadkel@gmail.com}}}
\affil[2]{Department of Computing and Mathematical Sciences, California Institute of Technology \protect\\ {\footnotesize\texttt{jslote@caltech.edu}}}
\affil[3]{Department of Mathematics, Michigan State University\protect\\
Hausdorff Center for Mathematics, University of Bonn\protect\\ {\footnotesize\texttt{volberg@msu.edu}}}
\affil[4]{Department of Mathematics, University of South Carolina\protect\\ {\footnotesize\texttt{haonanzhangmath@gmail.com}}}
\date{}

\usepackage{braket,bbm}
\usepackage{multirow,hhline}
\usepackage{tikz}
\usetikzlibrary{matrix,fit,decorations.pathreplacing,calc}

\newcommand{\E}{\mathop{{}\mathbb{E}}}
\newcommand{\R}{\mathbb{R}}
\newcommand{\Z}{\mathbb{Z}}
\newcommand{\C}{\mathbb{C}}
\newcommand{\BH}{\mathrm{BH}}
\newcommand{\un}{\mathbf{I}}

\newcommand{\ketbra}[2]{\ket{#1}\!\!\bra{#2}}

\newcommand{\renorm}{\Gamma\!}

\newcommand{\wt}[1]{\widetilde{#1}}

\newcommand{\bA}{\mathbf{A}}
\newcommand{\bB}{\mathbf{B}}
\newcommand{\bC}{\mathbf{C}}

\newcommand{\eps}{\varepsilon}

\newcommand{\bbt}{\mathbb{T}}

\newcommand{\al}{\alpha}

\newcommand{\cA}{\mathcal{A}}

\newcommand{\cW}{\mathcal{W}}

\newcommand{\Om}{\Omega}
\newcommand{\om}{\omega}

\newcommand{\go}{K}

\newcommand{\tr}{\textnormal{tr}}

\newtheorem{theorem}{Theorem}
\newtheorem{proposition}[theorem]{Proposition}

\newtheorem{lemma}[theorem]{Lemma}
\newtheorem{corollary}[theorem]{Corollary}

\theoremstyle{definition}
\newtheorem{definition}[theorem]{Definition}
\newtheorem{remark}[theorem]{Remark}

\newtheorem*{question*}{Question}

\newtheorem{claim}{Claim}

\definecolor{tealcolor}{RGB}{0, 127, 255}

\newcommand*{\claimproofname}{Proof of claim}
\newenvironment{claimproof}[1][\claimproofname]{\begin{proof}[#1]}{\end{proof}}

\newcommand{\iid}{\overset{\text{\textsf{iid}}}{\sim}}
\DeclareMathOperator{\supp}{supp}

\DeclareMathSymbol{\shortminus}{\mathbin}{AMSa}{"39}

\begin{document}

\maketitle

\begin{abstract}
Recent efforts in Analysis of Boolean Functions aim to extend core results to new spaces, including to the slice $\binom{[n]}{k}$, the hypergrid $[\go]^n$, and noncommutative spaces (matrix algebras).
We present here a new way to relate functions on the hypergrid (or products of cyclic groups) to their harmonic extensions over the polytorus.
We show the supremum of a function $f$ over products of the cyclic group $\{\exp(2\pi i k/\go)\}_{k=1}^\go$ controls the supremum of $f$ over the entire polytorus $(\{z\in\C:|z|=1\}^n)$, with multiplicative constant $C$ depending on $\go$ and $\deg(f)$ only.
This Remez-type inequality appears to be the first such estimate that is dimension-free (\emph{i.e.,} $C$ does not depend on $n$).

This dimension-free Remez-type inequality removes the main technical barrier to giving $\mathcal{O}(\log n)$ sample complexity, polytime algorithms for learning low-degree polynomials on the hypergrid and low-degree observables on level-$\go$ qudit systems.
In particular, our dimension-free Remez inequality implies new Bohnenblust--Hille-type estimates which are central to the learning algorithms and appear unobtainable via standard techniques.
Thus we extend to new spaces a recent line of work \cite{EI22, CHP, VZ22} that gave similarly efficient methods for learning low-degree polynomials on the hypercube and observables on qubits.

An additional product of these efforts is a new class of distributions over which arbitrary quantum observables are well-approximated by their low-degree truncations---a phenomenon that greatly extends the reach of low-degree learning in quantum science \cite{CHP}.
\end{abstract}

\section{Introduction}

\subsection{Motivation: quantum and classical low-degree learning}
\label{sec:intro}
Recall that any function $f:\{-1,1\}^n\to \R$ admits a unique Fourier expansion 
\[\textstyle f(x)=\sum_{S\subseteq[n]}\widehat{f}(S)\prod_{i\in S} x_i \quad \text{where}\quad \widehat{f}(S)=\E_{x\sim \{-1, 1\}^n}[f(x)\cdot\prod_{i\in S}x_i]\,.\]
We say $f$ is \emph{of degree $d$} ($\deg(f)\le d$) if for all $S$ with $|S|>d$, $\widehat{f}(S)=0$.

Low-degree learning---which we shall take to mean learning an $L_2$ approximation to $f$ with $\deg(f)\leq d$ and $|f|\le 1$ from uniformly random samples---is a fundamental task in computer science \cite{Odonnell}.
For a long time, the best polytime algorithm for low-degree learning had a sample complexity exponentially separated in $n$ from what was information theoretically-required (poly($n$) vs. $\mathcal{O}(\log n)$) \cite{LMN,fouriergrowth}.
But in 2022 Eskenazis and Ivanisvili closed the gap \cite{EI22}, achieving a sample complexity of $\mathcal{O}(\log n)$ for the first time.

Key to their argument is an observation about approximating sparse vectors.
Suppose one holds a vector $w\in \mathbb{R}^n$ that is an $\ell_\infty$ approximation to some unknown vector $v\in \mathbb{R}^n$ (say, $\|v-w\|_\infty\leq \eps$).
Then it could be that $\|v-w\|_2$ grows unboundedly with $n$.
However, if a constant $\ell_p$ bound on $v$ for $p<2$ is promised, it is possible to modify $w$ in a simple way to obtain a new approximation $\widetilde{w}$ to $v$ with $\|\widetilde{w}-v\|_2$ independent of $n$ (and controlled by $\eps$).
This is possible because the $\ell_p$ bound on $v$ implies $v$ is approximately sparse---and in particular that many of its coordinates are small in comparison to $\eps$.
Using this observation \cite{EI22} shows that zeroing-out coordinates in $w$ below a fixed threshold gives a suitable $\widetilde{w}$.

In the context of low-degree learning, $v$ is the vector of true low-degree Fourier coefficients of $f$, while $w$ is a vector of empirical Fourier coefficients collected through the familiar technique of Fourier sampling \cite{LMN,Odonnell}.
Modifying $w\mapsto\widetilde{w}$ as above we may form the approximate function $\widetilde{f}$ and with Plancherel's theorem conclude $\|f-\widetilde{f}\|_2 = \|v-\widetilde{w}\|_2$ is small.

Thus the Eskenazis--Ivanisvili argument reduces learning low-degree polynomials to finding an $\ell_p, p<2$ bound on the Fourier coefficients of $f$.
Inequalities of this kind are known as Bohnenblust--Hille-type inequalities, and the state of the art for polynomials over the hypercube was proved recently in \cite{DMP}:
\begin{theorem}[Hypercube Bohnenblust--Hille \cite{DMP}]
    \label{thm:hypercube-BH}
    For any $d\ge 1$, there exists a constant $C_d>0$ such that for all $n \ge 1$ and all $f:\{-1,1\}^n\to \R$ with $\deg(f)\le d$, we have
    \[\|\widehat{f}\|_{\frac{2d}{d+1}}:= \Big(\sum_{S\subseteq [n]}|\widehat{f}(S)|^{\frac{2d}{d+1}}\Big)^{\frac{d+1}{2d}}\leq C_d \|f\|_\infty\,.\]
    Moreover, there exists a universal constant $c>0$ such that $C_d\le c^{\sqrt{d\log d}}$.
\end{theorem}
\noindent
For our purposes the critical feature of Theorem \ref{thm:hypercube-BH} is the dimension-free-ness of the estimate.
Bohnenblust--Hille (BH) inequalities were originally proved for analytic polynomials over the polytorus $\bbt^n=\{z\in\C:|z|=1\}^n$ and have a long history in harmonic analysis
\cite{DS}.

One might ask if a similar approach to low-degree learning could work in the quantum world.
Quantum observables (Hermitian operators) on a system of $n$ qubits admit a ``Fourier-like'' decomposition 
\[
\cA = \sum_{\alpha\in\{0,1,2,3\}^n}\widehat{\cA}(\alpha)\sigma_\alpha \quad \text{where} \quad \sigma_\alpha = {\textstyle\bigotimes_{i=1}^n} \sigma_{\alpha_i}
\quad\text{and}\quad \widehat{\cA}(\alpha)=2^{-n}\mathop{\mathrm{tr}}[\cA\cdot\sigma_\alpha]\,.\]
Here $\sigma_0$ is the  2-by-2 identity matrix and $\sigma_i, 1\leq i \leq 3$ are the Pauli matrices
\[
\sigma_1=\begin{bmatrix}
0 & 1\\1 & 0
\end{bmatrix},\quad
\sigma_2=\begin{bmatrix}
0 & -i\\i & 0
\end{bmatrix},\quad
\sigma_3=\begin{bmatrix}
1 & 0\\0 & -1
\end{bmatrix}.\]
Defining $|\alpha|$ to be the number of nonzero entries in $\alpha$, we say $\cA$ is of degree $d$ if for all $\alpha$ with $|\alpha|>d$ we have $\widehat{\cA}(\alpha)=0$.
It was recently identified \cite{RWZ22,CHP, VZ22} that this analogy is close-enough to the Boolean case that the Eskenazis--Ivanisvili approach goes through for quantum observables as well, provided a BH-type inequality for Pauli decompositions exists, which was also proved:

\begin{theorem}[Qubit Bohnenblust--Hille \cite{CHP, VZ22}]
    \label{thm:qubit-BH}
    Suppose that $d\ge 1$. Then there exists a constant $C_d>0$ such that for all $n\ge 1$ and all $\cA$ on $n$ qubits of degree $d$, we have
    \[\|\widehat{\cA}\|_{\frac{2d}{d+1}}:=\Big(\sum_{\alpha\in \{0,1,2,3\}^n}|\widehat{\cA}(\alpha)|^{\frac{2d}{d+1}}\Big)^{\frac{d+1}{2d}}\leq C_d\|\cA\|_{\textnormal{op}}\,,\]
    where $\|\cA\|_{\textnormal{op}}$ denotes the operator norm.
\end{theorem}
An additional feature of low-degree learning specific to the quantum setting shown by \cite{CHP} is that when averaging over certain distributions of input states, low-degree truncations are good approximations of arbitrary quantum observables---even ones corresponding to exponential-time quantum computations.
This means the low-degree learning algorithm can perform well in predicting arbitrarily-complex quantum processes with respect to these distributions, a phenomenon that stands in stark contrast to the classical case.

\subsection{From $\Z^n_2$ to $\Z^n_\go$ and from qubits to qudits}

It is natural to ask whether these algorithms rely on special properties of the hypercube or qubit systems, or whether they extend to larger spaces.
In this paper we provide an affirmative answer, extending these classical and quantum learning results to \mbox{(tensor-)product} spaces of arbitrary local size $\go\geq 2$.

Classically, this extension works as follows.
We shall consider complex-valued functions $f:\Z_\go^n\to\C$, where $\Z_\go=\{0,1\dots, \go-1\}$ is the cyclic group of order $\go$.
Then each $f$ has a unique Fourier expansion:
\[f(x) = \sum_{\alpha\in\{0,1,\ldots,\go-1\}^n} \widehat{f}(\alpha)\prod_{j=1}^n\omega^{\alpha_j x_j}\quad \text{with}\quad \omega:=e^{\frac{2\pi i}{\go}}\,,\]
where $|\alpha| := \sum_j\alpha_j$. We say $f$ is \emph{of degree $d$} if $\widehat{f}(\alpha)=0$ for all $\alpha$ with $|\alpha|>d$.
Ultimately, we obtain the following algorithm.
\begin{theorem}[Cyclic Low-degree Learning]
    \label{thm:cyclic-learning}
    Let $f:\Z_\go^n\to \mathbb{D}$ be of degree $d$.
    Then with $(\log \go)^{\mathcal{O}(d^2)}\log(n/\delta)\eps^{-d-1}$ independent random samples $(x,f(x))$, $x\sim\mathcal{U}(\Z_\go^n)$, we may with confidence $1-\delta$ construct in polynomial time a function $\widetilde{f}:\Z_\go^n\to\C$ with $\|f-\widetilde{f}\|_2^2\leq \eps$.
\end{theorem}
\noindent(Compare with the naive Fourier sampling algorithm which would require $\mathrm{poly}(n)$ samples.)
Here the $L_2$ norm $\|\cdot\|_2$ is defined with respect to the uniform probability measure on $\Omega_\go^n$.
Our efforts in this direction are in the theme of generalizing Analysis of Boolean Functions results to more-general product spaces, \emph{e.g.}, \cite{BGKKL}.

It could be argued that the quantum case of generalized low-degree learning is even more important, both for the study of fundamental physics via quantum simulation (\emph{e.g.}, \cite{quant-sim,PhysRevLett.129.160501}) and in the operation and validation of quantum computers.
In both contexts, gains in efficiency are possible when the underlying hardware system is composed of higher-dimensional subsystems, sometimes carrying an algorithm from theoretical fact to practical reality in the NISQ era \cite{PhysRevLett.129.160501}---and this benefit may very well remain as quantum computing advances.
Such systems are called \emph{multilevel system}, or \emph{qudit} quantum computers \cite{quditsurvey}.
While the qu\emph{b}it case gives a conceptual sense of the possibilities for learning on qudit systems, it is practically important to derive guarantees and algorithms that work directly in the native dimension of the quantum system.
In so doing we also establish new distributions under which arbitrary quantum processes are well-approximated by low-degree ones (including new distributions for qubit systems beyond those identified in \cite{CHP}).

\begin{theorem}[Qudit Observable Learning, Informal]
    \label{thm:qudit-learning}
    Let $\cA$ be any (not necessarily low-degree) bounded quantum observable on $\mathcal{H}_\go^{\otimes n}$; \emph{i.e.,} on $n$-many $\go$-level qudits.
    Then we may via random sampling construct in polynomial time an approximate observable $\widetilde{\cA}$ such that for a wide class of distributions $\mu$ on states $\rho$,
    \[\E_{\rho\sim\mu}\big|\tr[\cA\rho]-\tr[\widetilde{\cA}\rho]\big|^2\leq \eps\,.\]
    The samples are of the form $(\rho, \tr[\cA \rho])$ for $\rho$ drawn from the uniform distribution over a certain set of product states.
    Moreover, the number of samples $s$ required to achieve the guarantee with confidence $1-\delta$ is
    \[s\leq \mathcal{O}\Big(\log(\tfrac{n}{\delta})\,C^{\log^2(1/\epsilon)}\go^{3/2}\|\cA^{\leq t}\|^{2t}_\mathrm{op}\Big)\,.\]
    The quantity $\|\cA^{\leq t}\|_\mathrm{op}$ is the operator norm of a degree-$t$ truncation of $\cA$, for $t$ roughly $\log(1/\epsilon)$.
\end{theorem}
\noindent There is little prior work on learning qudit observables, but earlier polytime algorithms for learning qubits required at least $\mathcal{O}(n\log(n))$ samples to complete a comparable task \cite{Huang2020}.
Theorem \ref{thm:qudit-learning} is proved by combining a low-degree learning algorithm for qudits with a qudit extension of the low-degree approximation lemma of Huang, Chen and Preskill \cite[Lemma 14]{CHP}.
The distributions $\mu$ admitting this construction are studied in Section \ref{sec:learning-arb}.

\subsection{Proof ideas: a new Remez-type inequality}
The extensions described above essentially amount to proving new Bohnenblust--Hille-type inequalities in the associated spaces.
Here we briefly describe how these are obtained and introduce our Remez inequality.

There are multiple natural generalizations of the qubit BH inequality (Theorem \ref{thm:qubit-BH}), and we pursue results for operator expansions in both the Heisenberg--Weyl and the Gell-Mann bases (these choices are explained and defined in Section \ref{sec:BHineq}).
We take inspiration from the technique of \cite{VZ22} to reduce these noncommutative BH inequalities to BH inequalities over classical, commutative spaces.
In the case of the Gell-Mann basis, we are able to reduce to the Hypercube BH, so we are done.
However, in the very natural and useful Heisenberg--Weyl basis (composed of clock \& shift operators), the eigenvalues of the corresponding matrices are the $\go$\textsuperscript{th} roots of unity.
Therefore it is natural to reduce to a BH inequality over products of cyclic groups---the same inequality needed for classical cyclic low-degree learning.

So in this way an important version of the qudit BH inequality dovetails with the BH inequality for cyclic groups and hypergrid learning.
The cyclic group BH inequality appears to be previously unstudied, despite it being the interpolating case between the hypercube case ($\go=2$) and the original polytorus case ($\go=\infty$).
One quickly discovers why, however: a proof by the ``standard recipe'' for BH inequalities (\textit{\`a la} \cite{DMP}, \cite{BPS}, \cite{DGMS}, \cite{CHP}) will not work here.

Let us sketch the difficulty.
At a very coarse level, BH inequalities for degree-$d$ polynomials on some product space $X^n$ are typically proved in these steps \cite{DS}:
\begin{enumerate}
    \item Symmetrization: express $f$ as the restriction of a certain symmetric $d$-linear form $L_f$ to the diagonal $\Delta:=\{(z,\ldots, z): z\in X^n\}$; that is, $f(z)=L_f(z,\ldots, z)$.
    \item BH for multilinear forms: bound the $\ell_{2d/(d+1)}$ norm of the coefficients of $L_f$ (which are directly related to the coefficients of $f$) by the supremum norm of $L_f$ over $(X^n)^d$.
    This step is rather involved and includes several estimates, manipulations, and an application of hypercontractivity and Khinchine's inequality.
    \item Polarization: estimate the supremum norm of $L_f$ on its entire domain $(X^n)^d$ by the supremum over $\Delta$; that is,
    \[\|L_f\|_{(X^n)^d}\lesssim\|L_f\|_{\Delta}=\|f\|_{X^n}\,,\]
    where $\|\cdot\|_E$ denotes the supremum norm over some space $E$.
\end{enumerate}
When adapting this proof structure to cyclic groups of order $2<\go<\infty$, the main point of failure is in step three, polarization.
In both the polytorus and hypercube cases, one uses Markov--Bernstein-type inequalities to obtain the intermediate inequality
\[\|L_f\|_{(X^n)^d}\lesssim\|f\|_{\mathrm{conv}(X)^n},\]
where $\mathrm{conv}(X)$ denotes the convex hull of $X$.
The passage from $\mathrm{conv}(X)$ to $X$ is then immediate for the polytorus by the maximum modulus principle ($\|f\|_{\mathbb{D}^n}=\|f\|_{\bbt^n}$) and for the hypercube by multilinearity ($\|f\|_{[-1,1]^n}= \|f\|_{\{-1,1\}^n}$).
But there is no such easy fact in the setting of the multiplicative cyclic group $\Om_\go:=\{e^{2\pi i k/\go}:k=0,\ldots, \go-1\}$ with $2<\go<\infty$ because $\Om_\go$ is not the entire boundary of $\mathrm{conv}(\Om_\go)$.
Even for $n=1$ and $\go=3$ one can construct example $f$'s for which $\|f\|_{\mathrm{conv}(\Om_\go)^n}>\|f\|_{\Om_\go^n}$.

Indeed, it was not at all clear that $\|f\|_{\Om_\go^n}$ should provide any reasonable control over $\|f\|_{\mathrm{conv}(\Om_\go)^n}$, let alone a bound with constant independent of dimension.
As a resolution, we here provide a new way to relate the supremum norm of a polynomial $f$ over $\Om_\go^n$ to its supremum norm over $\bbt^n$.
\begin{theorem}
\label{thm:remez}
    Fix $\go\ge 2$. Let $f$ be an $n$-variate degree-$d$ polynomial of individual degree at most $\go-1$.
    Then 
    \[\|f\|_{\bbt^n}\leq (\mathcal{O}(\log \go))^d\|f\|_{\Om_\go^n}\,.\]
\end{theorem}
This can be seen as a sort of maximum principle for $\Om_\go^n$ because we may conclude
\[\|f\|_{\mathrm{conv}(\Om_\go)^n}\leq \|f\|_{\mathbb{D}^n}=\|f\|_{\bbt^n}\lesssim\|f\|_{\Om_\go^n}.\]
However, with Theorem \ref{thm:remez} in hand there is no need to repeat any of the steps listed above.
The original Bohnenblust--Hille inequality for the polytorus \cite{BH} states
\[\|\widehat{f}\|_{\frac{2d}{d+1}} \leq C^{\sqrt{d\log d}}\|f\|_{\bbt^n}\,\]
for any degree-$d$ analytic polynomial $f$.
So we immediately obtain the cyclic group BH:

\begin{corollary}[Cyclic Group BH]
\label{cor:bh cyclic groups}
    Let $f$ be an $n$-variate degree-$d$ polynomial of individual degree at most $\go-1$.
    Then
    \begin{equation}
    \label{ineq:cyclic-BH}\|\widehat{f}\|_{\frac{2d}{d+1}}\leq (\mathcal{O}(\log \go))^{d+\sqrt{d\log d}}\|f\|_{\Om_\go^n}\,.
    \end{equation}
\end{corollary}

Remez-type inequalities bound the supremum of a low-degree polynomial $f$ over some space $X$ by the supremum of $f$ over some subset $Y\subseteq{X}$.
In this sense Theorem \ref{thm:remez} is a discrete Remez-type inequality for the polytorus; moreover it appears to be the first discrete multidimensional Remez inequality with a dimension-free constant (\emph{c.f.} \cite{Yomdin2011} and references therein).

We also remark that Theorem \ref{thm:remez} may be improved in certain regimes (when $d\ll \go$, or when $\go$ is a very composite integer), as well as readily extended to $L_p\to L_p$ comparisons for general $p$ and to spaces with less structure than $\Om_\go^n$.
These extensions are to appear elsewhere.
We believe Theorem \ref{thm:remez} could be of significant general interest,
as so much is already known about polynomials over $\bbt^n$.
Theorem \ref{thm:remez} provides a bridge from discrete spaces back into classical harmonic analysis.

\subsection{Organization}
Section \ref{sec:remez} is a self-contained proof of the the dimension-free Remez Inequality.
In section \ref{sec:BHineq} we then obtain our qudit Bohnenblust--Hille inequalities.
In Section \ref{sec:learning} we use these results and a slightly generalized version of Eskenazis--Ivanisvili to give the learning algorithms of Theorem \ref{thm:qudit-learning} and Theorem \ref{thm:cyclic-learning}.
In Section \ref{sec:learning} we also study the probability distributions of qudit states that allow for accurate low-degree approximations of arbitrary quantum operators.

 \subsection*{Acknowledgements}
 O.K. is supported in part by a grant from the Israel Science Foundation (ISF Grant No. 1774/20), and by a grant from the US-Israel Binational Science Foundation and the US National Science Foundation (BSF-NSF Grant No. 2020643).
 J.S. is supported by Chris Umans' Simons Foundation Investigator Grant.
 A.V. is supported by NSF grant DMS 2154402 and by the Hausdorff Center of Mathematics, University of Bonn.
Part of this work was started while J.S., A.V., and H.Z. were in residence at the Institute for Computational and Experimental Research in Mathematics in Providence, RI, during the Harmonic Analysis and Convexity program.
It is partially supported by NSF DMS-1929284.

\section{A dimension-free Remez inequality}
\label{sec:remez}
Let $\bbt:=\{z\in\C:|z|=1\}$ and denote by $\|f\|_X$ the sup norm of any function $f:X\to \C$.
In this section we prove the key technical result of this work.
We remark that two proofs of Theorem \ref{thm:remez} are actually known; a very different argument was given by three of the authors in \cite{SVZremez}.
While the proof in \cite{SVZremez} is interesting for its own reasons, the argument below gives a better constant which is important for learning applications.


A natural approach to proving Theorem \ref{thm:remez} is to consider a specific maximizer $z\in\bbt^n$ of $|f|$ and approximate it by a linear combination of evaluations of $f$ at points in $\Omega_\go^n$.
We might begin with this lemma for a single coordinate:

\begin{lemma}
    \label{lem:dft}
    Suppose $z\in \bbt$.
    Then there exists $\bm c:= (c_0,\ldots, c_{\go-1})$ such that for all $k=0,1,\ldots, \go-1$,
    \[z^k=\sum_{j=0}^{\go-1}c_j(\omega^j)^k.\]
    Moreover, $\|\bm c\|_1\leq B\log(\go)$ for a universal constant $B$.
\end{lemma}
\begin{proof}
    Let $\omega = \exp( 2 \pi i/ \go)$. 
    The discrete Fourier transform (DFT) of the array $A = (1, z, \ldots, z^{\go-1})$ yields $\go$ complex numbers $\wt{c_0}, \ldots, \wt{c_{\go-1}}$ so that 
\[
	z^k = A_k = \frac{1}{\go} \sum_{j=0}^{\go-1} \wt{c_j} \omega^{j k}
\]
for all $k=0,\ldots, \go-1$. Using $c_j := \frac{1}{\go}\wt{c_j}$ we get
\begin{equation}
\label{eq:DFT}
	z^k = \sum_{j=0}^{\go-1} c_j \omega^{j k}.
\end{equation}

    Recall the DFT coefficients are given by
\[
	\wt{c_j} = \sum_{k=0}^{\go-1} A_k \omega^{-k j}.
\]
Since $A_k= z^k$ we have
\[
	\wt{c_j} = \sum_{k=0}^{\go-1} z^k \omega^{-k j} = \frac{1 - (z/ \omega^j)^\go}{1 - (z / \omega^j)}.
\]
By the triangle inequality,
\[
	|\wt{c_j}| \leq \min \left( \go, \frac{2}{|\omega^j - z|} \right).
\]
Using that the harmonic number $H_\go = \sum_{k=1}^{\go} 1/k$ satisfies $H_\go \leq \log(\go)+1$, it is elementary to see that we have
\[
	\sum_{j=0}^{\go-1} |\wt{c_j}| \leq B\go \log \go
\]
for $B$ a sufficiently large constant.
That is,
\[
	\|\bm c\|_1=\sum_{j=0}^{\go-1} |c_j| = \frac{1}{\go} \sum_{j=0}^{\go-1} |\wt{c_j}| \leq B\log(\go)\,.\qedhere
\]
\end{proof}
In a single coordinate, Lemma \ref{lem:dft} provides the desired inequality as follows.
With $z\in\bbt$ a maximizer of $|f(z)|$ we have
\begin{align*}
    \|f\|_{\bbt}=|f(z)| &= |\sum_{k=0}^d a_k z^k|
    =|\sum_{k=0}^d\sum_{j=0}^{\go-1}a_kc_j(\omega^j)^k|
    = |\sum_{j=0}^{\go-1}c_jf(\omega^j)|\\
    &\leq \|\bm c\|_1 \|f\|_{\Omega_\go}\leq C\log(\go)\|f\|_{\Omega_\go}\,.\tag{H\"older}
\end{align*}
However, in higher dimensions, repeating this argument coordinatewise introduces an exponential dependence on $n$.
We circumvent this by taking a probabilistic view of the foregoing display: the sum over $j$ can be interpreted as an expectation over a (complex-valued) measure on $\Omega_\go$.
When it is repeated in several dimensions, this is like taking an expectation over $n$ independent random variables.
The key insight is that this independence is more than we need: by correlating the random variables, we save on randomness (which reduces the multiplicative constant) while retaining control of the error.

\begin{lemma}
    \label{lem:remez-main-idea}
    Let $f$ be a degree-$d$ $n$-variate polynomial and $\bm z\in \bbt^n$.
    Then there is a univariate polynomial $p=p_{f,\bm z}$ such that for any positive integer $m$ there are (dependent) random variables $R,\bm W$ taking values in $\Omega_4$ and $\Omega_\go^n$ respectively such that
    \begin{equation}
        \label{eq:main-idea}
        f(\bm z)=D^m\E_{R,\bm W}[Rf(\bm W)]+p(1/m)\,.
    \end{equation}
    Moreover, $p$ has $\deg(p)<d$ and zero constant term, and $D=D(K)$ is a universal constant depending on $K$ only.
\end{lemma}

Lemma \ref{lem:remez-main-idea} is the crux of our argument and we are not aware of a similar statement in the literature.
Theorem \ref{thm:remez} follows quickly, though it is interesting to note that instead of clearing the error term by taking $m\to\infty$ (which would indeed make $p(1/m)\to 0$ but also send $D^m\to\infty$), we will end up using \emph{algebraic} features of $p$ (namely, low-degree-ness) to remove it.
But first, the lemma:

\begin{proof}[Proof of Lemma \ref{lem:remez-main-idea}]
    We will argue Lemma \ref{lem:remez-main-idea} for $f(\bm z)=\bm z^\alpha$, a monomial of degree at most $d$.
    The claim extends to general degree-$d$ $f$ by linearity.
    
    We begin by examining a single coordinate with the aim of rewriting Lemma \ref{lem:dft} in a probabilistic form.
    To that end, we first decouple the angle and magnitude information of the $c_j$'s.
    Fix $z\in\bbt$ and let $c_j$ be as in Lemma \ref{lem:dft}.
    We may write a decomposition
    \[c_j=1\cdot c_j^{(0)}+i\cdot c_j^{(1)}+(\shortminus 1)\cdot c_j^{(2)}+(\shortminus i)\cdot c_j^{(3)}=\sum_{s=0}^3i^s\cdot c_j^{(s)}\,,\]
    with all $c_j^{(s)}\in\R^{\geq0}$ and $c_j^{(0)}c_j^{(2)}=c_j^{(1)}c_j^{(3)}=0$.
    This can be done for all $j$ so that, with $C:=B\log \go$ from Lemma \ref{lem:dft},
    \begin{equation}
    \label{eq:c-bound}
    \|\bm{c}^{(s)}\|_1\leq C
    \end{equation}
    is satisfied for each $s\in\{0,1,2,3\}$, where $\bm c^{(s)}=(c_1^{(s)},\ldots, c_n^{(s)})$.
    So we have for all $k=0,\ldots, \go-1$,
    \[z^k=\sum_{j=0}^{\go-1}\sum_{s=0}^3i^s \cdot c_j^{(s)}\cdot (\omega^j)^k\,.\]
    We now rewrite the sum in Lemma \ref{lem:dft} in probabilistic form.
    
    Put $D=4C+1$ and define $r:[0,D]\to\C$ by
    \[
    r(t)=\begin{cases}
        1 & 0\leq t\leq C+1,\\
        i & C+1 < t\leq 2C+1,\\
        -1 & 2C+1 < t\leq 3C+1,\\
        -i & 3C+1 < t\leq 4C+1=D\,.
    \end{cases}
    \]
    Also define a piecewise-constant function $w:[0,D]\to\Om_\go$ as follows.
    Consider any collection of disjoint intervals $I^{(s)}_j,0\le j\le \go-1,0\le s\le 3$ such that
   \[
   I_j^{(s)}\subset[0,D],\qquad s\in\{0,1,2,3\}, j\in\{0,1,\ldots, \go-1\}
   \] and for each $s$ and $j$, $I_{j}^{(s)}\subseteq[sC+1,(s+1)C+1]$ and $|I_j^{(s)}|=c_j^{(s)}$.
    Disjointness is possible because for each $s$,
    \[|[sC+1,(s+1)C+1]|=C\geq \textstyle\sum_{j=0}^{\go-1}c_j^{(s)}\]
    by \eqref{eq:c-bound}.
    Now assign $w(I_{j}^{(s)})=\omega^j$ and in the remaining region of $[0,D]$ (that is, ${[0,D]\backslash\sqcup_{s,j}I_j^{(s)}}$) let $w$ take on each element of $\Om_\go$ with in equal amount (w.r.t. the uniform measure).

    \begin{claim}
    \label{cl:univariate-identity}
    Let $T$ be sampled uniformly from $[0,D]$.
    Then for all $k=0,1,\ldots, \go-1$,
    \begin{equation}
        \label{eq:zk-integral}
        z^k=D\E_{T}[r(T)w(T)^k]\,.
    \end{equation}
    \end{claim}
    \begin{claimproof}[Proof of Claim \ref{cl:univariate-identity}]
        Let us begin with $k=0$, which simplifies to
    \begin{equation}
        D\E_{T}[r(T)]=1\,.
    \end{equation}
    This can be seen by direct computation:
    \[\E_{T}[r(T)]=\frac{1}{D}\left(1+1\cdot C + i\cdot C + (\shortminus1)\cdot C + (\shortminus i)\cdot C\right) = \frac{1}{D}\,.\]
    For $k\geq 1$, consider the joint distribution of $(r(T), w(T)^k)$, whose product appears in \eqref{eq:zk-integral}.
    Fix $s \in \{0,1,2,3\}$, and condition on $r(T) = i^s$.
    The conditional distribution of $w(T)$ has two parts.
    One part, corresponding to $\sqcup_j I_j^{(s)}$, has $w(T)=\omega^j$ over $I^{(s)}_j$ with the probability $\Pr[ r(T)=i^s \wedge w(T) = \omega^j ]$ equal to $c_j^{(s)} / D$, while the other has $w(T)$ uniformly distributed in $\Om_\go$.
    The latter part contributes $0$ to the expectation $\E[ r(T) w(T)^k ]$, since $\sum_{j=0}^{\go-1} (\omega^j)^k = 0$ for $k=1,2,\ldots, \go-1$.
    The former part contributes
	\[
		i^s \cdot \sum_{j=0}^{\go-1} \frac{c_j^{(s)}}{D} \omega^{j k}.
	\]
	Summing this display over $s\in \{0,1,2,3\}$ and rearranging, we get that
	\[
		\E[ r(T) w(T)^k ] = \sum_{j=0}^{\go-1} \frac{c_j}{D} (\omega^j)^k=\frac{1}{D}z^k,
	\]
    completing proof of \eqref{eq:zk-integral}.
    \end{claimproof}

    \medskip
    
    We return to the multivariate setting.
    Fix $\bm z:=(z_1,\ldots, z_n)\in \bbt^n$ and define the functions $w_1,\ldots, w_n$ corresponding to the above construction applied to each coordinate $z_1,\ldots, z_n$.
    If each coordinate were to receive a fresh copy of $T$ this would lead to an identity with exponential constant:
    \[{\bm z}^\alpha=D^n\E_{\substack{T_\ell\,\iid \,T,\\1\leq\ell\leq n}}\big[\textstyle\prod_{\ell=1}^n r(T_\ell)w_\ell(T_\ell)^{\alpha_\ell}\big].\]
    Instead, we consider only $m$ independent copies of $T$: $T_1,\ldots, T_m\iid\mathcal{U}[0,D]$.
    The decision of which coordinates are integrated with respect to which $T_\ell$ is also made randomly, via a uniformly random function $P:[n]\to[m]$.
    We finally arrive at the definitions of $R$ and $\bm W$:
    \begin{align*}
        R &:= \prod_{\ell=1}^m R_\ell \quad \text{with} \quad R_\ell:=r(T_\ell), 1\leq\ell\leq m\\
        \bm W &:=\Big(w_1\big(T_{P(1)}\big),w_2\big(T_{P(2)}\big), \ldots, w_n\big(T_{P(n)}\big)\Big)=:\left(W_1,\ldots,W_n\right)\,.
    \end{align*}
    
    When $P$ is injective on $\supp(\alpha)$, we easily achieve the smaller constant.
    \begin{claim}
    \label{cl:fullpart}
    Consider $m\geq |\supp(\alpha)|$.
    Then
    \[\E_{R,\bm W}[R\cdot \bm W^\alpha\mid P \text{ is injective on }\supp(\alpha)]=D^{-m}\bm z^\alpha\,.\]
    \end{claim}
    \begin{claimproof}[Proof of Claim \ref{cl:fullpart}]
        It suffices to prove this for an arbitrary projection $\widetilde{P}$ that is injective on $\supp(\alpha)$.
    Consider the partition of $[n]$ given by ${\widetilde P}^{-1}([m])$ and write $S_\ell={\widetilde P}^{-1}(\ell)$ for $\ell\in[m]$.
    By independence, the expectation splits over these $S_\ell$'s:
    \begin{equation}
        \label{eq:exp-split}
        \E_{R,\bm W}[R\cdot \bm W^\alpha\mid P =\widetilde{P}]=\textstyle\prod_{\ell=1}^m\E\left[R_\ell\prod_{k\in S_\ell}W_k^{\alpha_k}\right].
    \end{equation}
    Because $\widetilde{P}$ is injective on $\supp(\alpha)$, every $S_\ell$ contains one or zero elements of $\supp(\alpha)$.
    By Claim \ref{cl:univariate-identity}, in the latter case we have
    \[\E\left[R_\ell\textstyle\prod_{k\in S_\ell}W_k^{\alpha_k}\right] = \E[R_\ell]=\frac1D,\]
    and in the former case we have
    \[\E\left[R_\ell\textstyle\prod_{k\in S_\ell}W_k^{\alpha_k}\right] = \E[R_\ell W_j^{\alpha_j}] = \frac{1}{D}z_j^{\alpha_j},\]
    for the specific $j$ for which $\{j\}=S_\ell\cap\supp(\alpha)$.
    Substituting these observations into \eqref{eq:exp-split} completes the argument.
    \end{claimproof}

    When $P$ is not injective, we still have some control.
    Let $\mathcal{S}=\{S_j\}$ be a partition of $\supp(\alpha)$.
    We say $P$ \emph{induces} $\mathcal{S}$ if 
    \[\{P^{-1}(j)\cap \supp(\alpha):j\in [m]\}=\mathcal{S}\,.\]
    
    \begin{claim}
    \label{cl:E(S)}
    For each partition $\mathcal{S}$ of $\supp(\alpha)$ there is a number $E(\mathcal{S})$ such that for all $m \geq |\mathcal{S}|$,
    \[\E_{R,\bm W}[R\cdot \bm W^\alpha\mid P \text{ induces } \mathcal{S}]=D^{-m}E(\mathcal{S}).\]
    \end{claim}
    \begin{claimproof}[Proof of Claim \ref{cl:E(S)}]
        Condition again on a specific $\wt P$ that induces $\mathcal{S}$.
        There are two types of $\ell\in[m]$: those that $\bm W^\alpha$ depends on (that is, $\widetilde P(\supp(\alpha))$), and those that only $R$ depends on.
        Call these sets $L=\widetilde P(\supp(\alpha))$ and $L^c$ respectively.
        Then by independence of the $T_\ell$'s,
    \begin{align*}
    \E_{R,\bm W}[R\cdot \bm W^\alpha\mid P=\wt P]&=\E_{R,\bm W}\textstyle[\left(\prod_{\ell\in L^c}R_\ell\right)\left(\prod_{\ell\in L}R_\ell\right)\cdot \bm W^\alpha\mid P = \wt P]\\
    &= D^{-m+|\mathcal{S}|}\underbrace{\E_{R,\bm W}\textstyle[\left(\prod_{\ell\in L}R_\ell\right)\cdot \bm W^\alpha\mid P = \wt P]}_{*}\,.
    \end{align*}
    We observe that the expectation ($*$) does not depend on the specific $\wt P$ inducing $\mathcal{S}$, nor on $m$.
    Thus we may define $E(\mathcal{S})$ by setting $D^{-|\mathcal{S}|}E(\mathcal{S})$ equal to ($*$).
    \end{claimproof}

    To summarize claims \ref{cl:fullpart} and \ref{cl:E(S)},
    we have that for all partitions $\mathcal{S}$ of $\supp(\alpha)$ there is a number $E(\mathcal{S})$ such that for all $m \geq |\mathcal{S}|$,
    \[
    \E[R\cdot\bm W^\alpha|P\text{ induces }\mathcal{S}]= D^{-m}E(\mathcal{S}).
    \]
    And using $\mathcal{S}^*$ to denote the singleton partition $\big\{\{j\}\big\}_{j\in\supp(\alpha)}$, we additionally have $E(\mathcal{S}^*)=\bm z^\alpha$.
    
    \medskip

    Now we consider the unconditional expectation $\E[R\cdot \bm W^\alpha]$ with $P\sim\mathcal{U}({[m]}^{[n]})$.
    Simple combinatorics give that for all partitions $\mathcal{S}$ and all $m\geq 1$, with $s=|\mathcal{S}|$,
    \[\Pr[P \text{ induces } \mathcal{S}]=\frac{m(m-1)\cdots(m-s+1)}{m^{|\supp(\alpha)|}}=:
    \begin{cases}
        1+q_{s}\big(\tfrac{1}{m}\big) & \text{if } s=|\supp(\alpha)|\\
        q_{s}\big(\tfrac{1}{m}\big) & \text{if } s<|\supp(\alpha)|\,,
    \end{cases}\]
    for polynomials $q_{s}$ with zero constant term and $\deg(q_{s})<d$.

    Of course $P$ can only induce $\mathcal{S}$ for $|\mathcal{S}|\leq m$, so by the law of total probability,
    \[\E_{R,\bm W}[R\cdot \bm W^\alpha]=\sum_{\mathcal{S},|\mathcal{S}|\leq \min(m,|\supp(\alpha)|)}\E[R\cdot \bm W^\alpha\mid P \text{ induces } \mathcal{S}\,]\Pr[P \text{ induces } \mathcal{S}]\,.\]
    Consider first the case $m \geq |\supp(\alpha)|$.
    We obtain
    \begin{align}
    \label{eq:fullcase-mbig}
        \E_{R,\bm W}[R\cdot \bm W^\alpha] &= \sum_{\mathcal{S}}\E[R\cdot \bm W^\alpha\mid P \text{ induces } \mathcal{S}\,]\Pr[P \text{ induces } \mathcal{S}]\nonumber\\
        &= D^{-m} E(\mathcal{S}^*) \Big(1+q_{|\supp(\alpha)|}\big(\tfrac1m\big)\Big)+\sum_{\mathcal{S}, |\mathcal{S}|<|\supp(\alpha)|}D^{-m}E(\mathcal{S})\cdot q_{|\mathcal{S}|}\big(\tfrac1m\big)\nonumber\\
        &= D^{-m}\Big[\bm z^\alpha + \sum_{\mathcal{S}}E(\mathcal{S})\cdot q_{|\mathcal{S}|}\big(\tfrac1m\big)\Big]\,.
    \end{align}
    Now when $m < |\supp(\alpha)|$, we combine the fact that $\Pr[P \text{ induces } \mathcal{S}]=0$ for $|\mathcal{S}|>m$ with the definition of $q_s$ to see
    \begin{align}
    \label{eq:fullcase-msmall}
        \E_{R,\bm W}[R\cdot\bm W^\alpha] &= 0+\sum_{\mathcal{S}, |\mathcal{S}|\leq m}\E[R\cdot \bm W^\alpha\mid P \text{ induces } \mathcal{S}\,]\Pr[P \text{ induces } \mathcal{S}]\nonumber\\
        &= \sum_{\mathcal{S},|\mathcal{S}|>m}D^{-m}E(\mathcal{S})\Pr[P \text{ induces } \mathcal{S}] + \sum_{\mathcal{S}, |\mathcal{S}|\leq m}D^{-m}E(\mathcal{S})\Pr[P \text{ induces } \mathcal{S}]\nonumber\\
        &= D^{-m}E(\mathcal{S}^*)\Big(1+q_{|\supp(\alpha)|}\big(\tfrac1m\big)\Big)\nonumber\\
        &\hspace{5em}+\sum_{|\supp(\alpha)|>|\mathcal{S}|>m}D^{-m}E(\mathcal{S})\cdot q_{|\mathcal{S}|}\big(\tfrac1m\big)+\sum_{m\geq |\mathcal{S}|}D^{-m}E(\mathcal{S})\cdot q_{|\mathcal{S}|}\big(\tfrac1m\big)\nonumber\\
        &= D^{-m}\Big[\bm z^\alpha + \sum_{\mathcal{S}}E(\mathcal{S})\cdot q_{|\mathcal S|}\big(\tfrac1m\big)\Big]\,.
    \end{align}
    Noting that \eqref{eq:fullcase-msmall} and \eqref{eq:fullcase-mbig} are identical, we rearrange to find
    \[\bm z^\alpha = D^m\E[R\cdot \bm W^\alpha] - \sum_{\mathcal{S}}E(\mathcal{S})\cdot q_{|\mathcal S|}\big(\tfrac1m\big), \]
    and the second part is in total a polynomial in $\tfrac1m$ with no constant term and degree $<d$.
\end{proof}

Finally, the error term $p\big(\tfrac1m\big)$ is removed by considering several values of $m$.

\begin{proof}[Proof of Theorem \ref{thm:remez}]
    Suppose there were some coefficients $a_m \in \C$ with $\sum_{m=1}^{d} a_m = 1$, so that for any polynomial $p$ of degree $< d$ and $p(0) = 0$ we would have
    \[
    	\textstyle\sum_{m=1}^{d} a_m p(\tfrac1m)=0.
    \] 
    We could then sum~\eqref{eq:main-idea} for $m=1, \ldots, d$, weighted by $a_m$, and get
    \[
    	f(z) = \sum_{m=1}^{d} a_m f(z) = \sum_{m=1}^{d} a_m D^{m} \E[ R_m f(W_m)] + \sum_{m=1}^{d} a_m p\big(\tfrac1m\big) = \sum_{m=1}^{d} a_m D^{m} \E[ R_m f(W_m)],
    \]
    where $R_m, W_m$ are those $R,\bm W$ from~\eqref{eq:main-idea} marked with explicit dependence on $m$.
    
    Well, these coefficients $a_m$ can be arranged, since the monomial vectors $(1/m^t)_{m=1,\ldots, d}$ for $t=0,\ldots, d-1$ are linearly independent (Vandermonde).
    Since always $|R_m| \leq 1$, we deduce
    \[
    	|f(z)| \leq \sum_{m=1}^{d} |a_m D^{m}| \cdot \|f\|_{\Omega_\go^n} \leq \frac{\max_{m=1}^{d} |a_m|}{1-1/D} \cdot D^{d} \|f\|_{\Om_\go^n}. 
    \]
    An explicit formula for the $a_m$'s is given by
    \[
    	a_{m} = (-1)^{d-m} \frac{m^d}{m! (d-m)!},
    \]
    and it is evident that $\max_{m=1}^{d} |a_m| \leq \exp(O(d))$, and specifically $\max_{m=1}^{d} |a_m| \leq \exp(1.28d)$.
    
    Without loss of generality, we may assume $D \geq 11$ thus $1/(1-1/D) \leq 1.1$, so we conclude
    \[
    	|f(z)| \leq
    	(4D) ^ {d} \|f\|_{\Om_\go^n} =
    	(4B \log(\go) + 4)^d \|f\|_{\Om_\go^n}.\qedhere
    \]
\end{proof}

\section{Qudit Bohnenblust--Hille inequalities}
\label{sec:BHineq}
Let 
\[
f(z) = \sum_\al c_\al z^\al=\sum_\al c_\al z^{\al_1}_1 \cdots z_n^{\al_n},
\]
where $\al=(\al_1, \dots, \al_n)$ are vectors of non-negative integers, all $c_\alpha$ are nonzero, and the total degree of polynomial $f$ is $d=\max_\al (\al_1+\dots+\al_n)$.
Here  $z$ can be all complex vectors in $\mathbb{T}^n=\{\zeta\in \C:|\zeta|=1\}^n$ or all sequences of $\pm1$ in Boolean cube $\{-1,1\}^n$. Bohnenblust--Hille-type of inequalities are the following
\begin{equation}
	\label{BHcom}
	\Big(\sum_\al |c_\al|^{\frac{2d}{d+1}}\Big)^{\frac{d+1}{2d}} \le C(d) \sup_{z}|f(z)|\,.
\end{equation}
The supremum is taken either over the torus $\mathbb{T}^n$ or, more recently, the Boolean cube $\{-1,1\}^n$.
In both cases this inequality is proven with constant $C(d)$ that is independent of the dimension $n$ and sub-exponential in the degree $d$.
More precisely, denote by $\textnormal{BH}^{\le d}_{\mathbb{T}}$ and $\textnormal{BH}^{\le d}_{\{\pm 1\}}$ the best constants in the Bohnenblust--Hille inequalities \eqref{BHcom} for degree-$d$ polynomials on $\mathbb{T}^n$ and $\{-1,1\}^n$, respectively. Then both $\textnormal{BH}^{\le d}_{\mathbb{T}}$ and $ \textnormal{BH}^{\le d}_{\{\pm 1\}}$ are bounded from above by $e^{c\sqrt{d\log d}}$ for some universal $c>0$ \cite{BPS,DMP}. 

One of the key features of this inequality \eqref{BHcom} is the dimension-freeness of $C(d)$. This, together with its sub-exponential growth phenomenon in $d$, plays an important role in resolving some open problems in functional analysis and harmonic analysis \cite{DGMS,BPS,DFOOS}. The optimal dependence of $\textnormal{BH}^{\le d}_{\mathbb{T}}$ and $ \textnormal{BH}^{\le d}_{\{\pm 1\}}$ on the degree $d$ remains open.

The qubit BH inequality, Theorem \ref{thm:qubit-BH}, has received two very different proofs.
In \cite{CHP} Huang, Chen and Preskill pursue a direct proof and notably develop a physically-motivated ``algorithmic'' procedure to prove the key step in BH-type arguments known as \emph{polarization}.
They achieve the dimension-free constant $C_d\leq \mathcal{O}(d^d)$.
Another proof approach appears in \cite{VZ22}, which works by reducing the qubit BH inequality to the hypercube BH inequality.
Let $\BH^{\leq d}_{M_2}$ denote the optimal constant in Theorem \ref{thm:qubit-BH} (where $M_2$ designates the 2-by-2 complex matrix algebra).
Then \cite{VZ22} showed $\BH^{\leq d}_{M_2}\leq 3^d\BH^{\leq d}_{\{\pm 1\}}\leq C^{\mathcal{O}(d)}$.

Pauli matrices are very special objects, being Hermitian, unitary, and anticommuting, and it was unclear whether the reduction approach in \cite{VZ22} could be extended to the qudit setting, where higher-dimensional generalizations of Pauli matrices are not so well-behaved.
In fact we succeed in extending the reduction argument to two bases for the complex matrix algebra $M_\go(\C)$ (tensors of which form the appropriate space for qudit systems) known as the (generalized) \emph{Gell-Mann basis} and the \emph{Heisenberg--Weyl basis} with the view to reduce to scalar BH inequalities.
They are orthonormal with respect to the normalized trace inner product $\frac{1}{\go}\tr[A^\dagger B]$, and are respectively Hermitian and unitary generalizations of the 2-dimensional Pauli basis.
Our proofs of these extensions reveal some pleasing features of the geometry of the eigenvalues of GM and HW matrices.

\begin{definition}[Gell-Mann Basis]
    Let $\go\ge 2$ and put $E_{jk}=\ket{e_j}\!\!\bra{e_k}, 1\le j,k\le \go$.
    The generalized Gell-Mann Matrices are a basis of $M_\go(\C)$ and are comprised of the identity matrix $\un$ along with the following generalizations of the Pauli matrices:
\begin{align*}
	\text{symmetric: }&& \bA_{jk} &= {\textstyle\sqrt{\!\frac{\go}{2}}}\big(E_{jk}+E_{kj}\big) & &\text{ for } 1\leq j<k\leq \go\\[0.3em]
	\text{antisymmetric: }&& \bB_{jk} &= {\textstyle\sqrt{\!\frac{\go}{2}}}\big(-iE_{jk}+iE_{kj}\big) & &\text{ for } 1\leq j < k\leq \go\\
	\text{diagonal: }&& \bC_m\, &= \renorm_{m}\left(\textstyle\sum_{k=1}^mE_{kk}-mE_{m+1,m+1}\right)  & &\text{ for } 1\leq m \leq \go-1,
\end{align*}
where $\renorm_{m} :=\sqrt{\!\frac{\go}{m^2+m}}$.
We denote
\[
\textnormal{GM(\go)}:= \{\un, \bA_{jk}, \bB_{jk}, \bC_m\}_{1\leq j<k\leq \go, 1\leq m \leq \go-1}\,.
\]
\end{definition}

An observable $\cA$ has expansion in the GM basis as
\[\cA = \sum_{\alpha\in \Lambda_\go^n}\widehat{\cA}(\alpha)M_\alpha=\sum_{\alpha\in \Lambda_\go^n}\widehat{\cA}(\alpha){\textstyle\bigotimes_{j=1}^n} M_{\alpha_j}\]
for some index set $\Lambda_\go$ (so $\{M_\alpha\}_{\alpha\in\Lambda_\go}=\mathrm{GM}(\go)$).
Letting $|\alpha|=|\{j:M_{\alpha_j}\neq \mathbf{I}\}|$, we say $\cA$ is of degree $d$ if $\widehat{\cA}(\alpha)=0$ for all $\alpha$ with $|\alpha|>d$.

We find the Gell-Mann BH inequality enjoys a reduction to the hypercube BH inequality on $\{-1,1\}^{n(\go^2-1)}$ and obtain the following.

\begin{theorem}[Qudit Bohnenblust--Hille, Gell-Mann Basis]\label{thm:GM}
	Fix any $\go\ge 2$ and $d\ge 1$. There exists $C(d,\go)>0$ such that for all $n\ge 1$ and GM observable $\cA\in M_\go(\C)^{\otimes n}$ of degree $d$, we have
	\begin{equation}\label{ineq:bh GM}
		\|\widehat{\cA}\|_{\frac{2d}{d+1}}\le C(d,\go)\|\cA\|_{\textnormal{op}}.
	\end{equation}
Moreover, we have $C(d,\go)\le \big(\tfrac32(\go^2-\go)\big)^d \textnormal{BH}^{\le d}_{\{\pm 1\}}$.
\end{theorem}
\noindent In particular, for $\go=2$ we recover the main result of \cite{VZ22} exactly.

\begin{definition}[Heisenberg--Weyl Basis]
Fix $\go\geq 2$ and let $\omega = \exp(2\pi i/\go)$.
Define the $\go$-dimensional clock and shift matrices respectively via
\begin{equation*}
	X \ket{j}= \ket{j+1},\qquad Z\ket{j} = \om^j \ket{j},\qquad \textnormal{for all} \qquad j\in \mathbb{Z}_\go.
\end{equation*}
Note that $X^\go=Z^\go=\un$. See more in \cite{AEHK}. 
Then the Heisenberg--Weyl basis for $M_\go(\C)$ is
\[
\textnormal{HW}(\go):=\{X^\ell Z^m\}_{\ell,m\in \{0,1,\ldots, \go-1\}}\,.
\]
\end{definition}
\noindent Any observable $A\in M_\go(\bC)^{\otimes n}$ has a unique Fourier expansion with respect to $\textnormal{HW}(\go)$ as well:
\begin{equation}
    \label{expHW}
	A=\sum_{\vec{\ell},\vec{m}\in \mathbb{Z}_\go^n}\widehat{A}(\vec{\ell},\vec{m})X^{\ell_1}Z^{m_1}\otimes \cdots \otimes X^{\ell_n}Z^{m_n},
\end{equation}
where $\widehat{A}(\vec{\ell},\vec{m})\in\C$ is the Fourier coefficient at $(\vec{\ell},\vec{m})$. 
We say that $A$ is \emph{of degree $d$} if $\widehat{A}(\vec{\ell},\vec{m})=0$ whenever 
\begin{equation*}
	|(\vec{\ell},\vec{m})|:=\sum_{j=1}^{n}(\ell_j+m_j)>d.
\end{equation*}
Here, $0\le \ell_j,m_j\le \go-1$.

Unlike in the GM expansion, HW Fourier coefficients may be complex-valued.
In fact, because the spectra of Heisenberg--Weyl matrices are the roots of unity, it is natural to pursue a reduction to a scalar BH inequality over $\Z_\go^n$---precisely the inequality needed for classical learning on functions on $\Z_\go^n$.
This reduction works when $\go$ is prime.

\begin{theorem}[Qudit Bohnenblust--Hille, Heisenberg--Weyl Basis]\label{thm:bh HW}
	Fix a prime number $\go\ge 2$ and suppose $d\ge 1$.
    Consider an observable $A\in M_\go(\C)^{\otimes n}$ of degree $d$.
    Then we have
	\begin{equation}
        \label{ineq:bh-hw}
		\|\widehat{A}\|_{\frac{2d}{d+1}}\le C(d,\go)\|A\|_{\textnormal{op}},
	\end{equation}
	with $C(d,\go)\le (\go+1)^d \textnormal{BH}_{\mathbb{Z}_\go}^{\le d}$.
\end{theorem}
\begin{table}
    \centering
    \makebox[\textwidth][c]{
    \renewcommand{\arraystretch}{1.5}
    \begin{tabular}{ |l|c|c||l |c |c |} 
     BH Const. & Best known bound & Source & BH Const. & Best known bound & Source\\
     \hline
     \multirow{2}{*}{\;\;$\BH^{\leq d}_{\{\pm1\}}$} & \multirow{2}{*}{$C^{\sqrt{d\log d}}$} & \multirow{2}{*}{\cite{DMP}} & $\;\;\BH^{\leq d}_{M_2}$ & $3^d\BH_{\{\pm1\}}^{\leq d}\leq C^{d}$& \cite{VZ22}\\
    \hhline{~~~||---}
     & & & $\;\;\BH^{\leq d}_{\mathrm{GM}(\go)}$ & $\big(\tfrac32(\go^2-\go)\big)^d \textnormal{BH}^{\le d}_{\{\pm 1\}}\leq \go^{C d}$ & Thm. \ref{thm:GM}\\
     \hline
     $\;\;\BH^{\leq d}_{\Z_\go}$ & $(C\log \go)^{d+\sqrt{d\log d}}$ & Cor. \ref{cor:bh cyclic groups} & $\;\;\BH^{\leq d}_{\mathrm{HW}(\go)}$ & $(\go+1)^d\cdot\BH_{\Z_\go}^{\leq d}\leq (C \go\log \go)^{C'd}$\;\textsuperscript{*} & Thm. \ref{thm:bh HW}\\
     \hline
     $\;\;\BH^{\leq d}_{\mathbb{T}}$ & $C^{\sqrt{d\log d}}$ & \cite{BPS} &  \multicolumn{3}{c }{} \\
     \cline{1-3}
    \end{tabular}
    }
    \renewcommand{\arraystretch}{1}
    \vspace{-2.5em}
    \begin{flushright}
    \begin{minipage}{8.5cm}
    \footnotesize
        $*$ : For $\go$ prime.
    \end{minipage}
    \end{flushright}
    
    \vspace{0.5em}
    \caption{Best known constants in Bohnenblust--Hille inequalities for (tensor-)product spaces at the time of this writing.
    Each BH inequality in the right half is proved via a reduction to the scalar BH inequality directly to its left.
    The results are accurate for all $K \geq 3$, and each appearance of $C$ and $C'$ is a different constant $>1$.
    For the bounds proved in this work, no prior bounds were known.}
    \label{fig:BH-table}
\end{table}

A summary of the Bohnenblust--Hille inequalities proved in this paper is provided in Table \ref{fig:BH-table}, where we denote the best constants in Eqs. \eqref{ineq:cyclic-BH}, \eqref{ineq:bh GM}, and \eqref{ineq:bh-hw} respectively by $\textnormal{BH}_{\mathbb{Z}_\go}^{\le d}$, $\textnormal{BH}_{\mathrm{GM}(\go)}^{\le d}$, and $\textnormal{BH}_{\mathrm{HW}(\go)}^{\le d}$.

\subsection{Qudit Bohnenblust--Hille in the Gell-Mann basis}

In this section we prove Theorem \ref{thm:GM} by reducing \eqref{ineq:bh GM} to the hypercube Bohnenblust--Hille inequality on $\{-1, 1\}^{n (\go^2-1)}$.
(The $\go=2$ case was done in \cite{VZ22}).

The central part of the reduction is a coordinate-wise construction of  density matrices $\rho(\boldsymbol{x})\in M_\go(\C)$ parametrized by $\boldsymbol{x}\in \{-1,1\}^{\go^2-1}=:H_\go$.
It will be convenient to partition the coordinates of
$\boldsymbol{x}$ as $\boldsymbol{x}=(x,y,z)\in \{-1, 1\}^{\binom{\go}{2}}\times \{-1, 1\}^{\binom{\go}{2}}\times \{-1, 1\}^{\go-1}$
with indices
\begin{align*}
	x=(x_{jk})_{1\le j<k\le \go}, \qquad y=(y_{jk})_{1\le j<k\le \go},\qquad 
 \text{and} \qquad z=(z_m)_{1\le m\le \go-1}\,.
\end{align*}

\begin{lemma}
	\label{action}
	For any $(x,y,z)\in H_\go$, there exists a positive semi-definite Hermitian
 matrix $\rho=\rho(x,y,z)$ with $\tr[\rho]=3\binom{\go}{2}$ such that for all $1\le j<k\le \go$ and $1\le m \le \go-1$,
	\begin{align}
    \label{action-xs}
	\tr [\bA_{jk} \rho(x, y, z)] &= {\textstyle\sqrt{\!\frac{\go}{2}}}x_{jk},\\
    \label{action-ys}
	\tr [\bB_{jk} \rho(x, y, z)] &= {\textstyle\sqrt{\!\frac{\go}{2}}}y_{jk},\\
    \label{action-zs}
	\tr [\bC_{m} \rho(x, y, z)] &= {\textstyle\sqrt{\!\frac{\go}{2}}}z_{m}.
	\end{align}
\end{lemma}
\begin{proof}

For $b\in\{-1,1\}$ and $1\le j<k\le \go$ consider unit vectors
\[
\al_{jk}^{(b)} = (e_j+be_k)/\sqrt{2}, \qquad \beta_{jk}^{(b)}=(e_j+b i e_k)/\sqrt{2}\,.
\]
These are respectively eigenvectors of $\bA_{jk}$ and $\bB_{jk}$ with eigenvalue $b\sqrt{\!\go/2\,}$. 
Now consider the density matrices, again for $b\in\{-1,1\}$ and $1\le j<k\le \go$,
\[
A_{jk}^{(b)}= \ketbra{\al_{jk}^{(b)}}{\al_{jk}^{(b)}} ,\qquad B_{jk}^{(b)}= \ketbra{\beta_{jk}^{(b)}}{\beta_{jk}^{(b)}}.
\]
Finally, define
\[
\rho=\rho(x, y, z) =\sum_{1\le j<k\le \go}  A_{jk}^{(x_{jk})} + \sum_{1\le j<k\le \go} B_{jk}^{(y_{jk})} + \sum_{m=1}^{\go-1}  \tfrac{z_m}{\sqrt{2\go}}\bC_m+ \tfrac{\go-1}{2}\cdot  \un\,.
\]
Observe $\rho$ is a positive semi-definite Hermitian matrix: each $A_{jk}^{(x_{jk})},B_{jk}^{(y_{jk})}$ are positive semi-definite Hermitian and the remaining summands form a diagonal matrix with positive entries.
Also, we have 
\begin{equation}
	\label{trace-rho}
	\tr \,\rho = \frac{\go(\go-1)}{2}+\frac{\go(\go-1)}{2}+0+\frac{\go(\go-1)}{2}=3\binom{\go}{2}\,.
\end{equation}

	Note for any $1\leq j<k\leq \go$ the anti-commutative relationship
	$\bA_{jk} \bB_{jk} + \bB_{jk}\bA_{jk}=0.$
    This implies that (see for example \cite[Lemma 2.1]{VZ22}) for any $b\in\{-1,1\}$,
    $
	\braket{\bA_{jk} \beta_{jk}^{(b)},\beta_{jk}^{(b)}} = \braket{\bB_{jk} \alpha_{jk}^{(b)},\alpha_{jk}^{(b)}}=0,
    $
    and thus
    \begin{equation*}
    \tr[\bA_{jk} B_{jk}^{(b)}]=\tr[\bB_{jk} A_{jk}^{(b)}]=0\,.
    \end{equation*}
	When $(j, k)\neq (j', k')$ then the operators ``miss'' each other and we get for all $b\in\{-1,1\}$
	\begin{equation*}
		\tr[\bA_{jk} B_{j'k'}^{(b)}]=\tr[\bB_{jk} A_{j'k'}^{(b)}]=\tr[\bA_{jk} A_{j'k'}^{(b)}]=\tr[\bB_{jk} B_{j'k'}^{(b)}]=0.
	\end{equation*}
    By orthogonality the remaining summands in $\rho$ contribute $0$ to $\tr(\bA_{jk}\rho),\tr(\bB_{jk}\rho)$.
    We conclude \eqref{action-xs} and \eqref{action-ys} hold.
    
	So far all follows more or less the path of \cite{VZ22}. A bit more surprising are the cancellations giving \eqref{action-zs}. For any $x,y\in\{-1,1\}^{\binom{K}{2}}$, we claim
	\begin{equation}
		\label{CjA}
		\tr \bigg[\bC_m \Big(\sum_{1\le j<k\le \go}    A_{jk}^{(x_{jk})}\Big)\bigg]=0 \qquad\text{and}\qquad
		\tr \bigg[\bC_m \Big(\sum_{1\le j<k\le \go}    B_{jk}^{(y_{jk})}\Big)\bigg]=0\,.
	\end{equation}
	Let us prove the first part of \eqref{CjA} with Figure \ref{fig:C-cancel-matrix} for reference.
	For a fixed $k>m+1$ we can immediately see that $\sum_{j=1}^{k-1} \tr [\bC_m A^{(x_{jk})}_{j k}] = \frac12\renorm_m(1+1+\dots+1 -m)=0$.
	We are left to consider the $j<k\le m$ summation and the $j\le m, k=m+1$ summation.
	The first one gives $\binom{m}{2}\renorm_m$, while the second one gives $\frac12m(1-m)\renorm_m = -\binom{m}{2}\renorm_m$, so their combined sum is again $0$.
    This proves the LHS of \eqref{CjA}; the RHS is argued identically.
    
	Now, the rest of $\rho(x, y, z)$ is $\sum_{m=1}^{\go-1}  \frac{z_m}{\sqrt{2\go}}\bC_m + \frac{\go-1}{2}\un$, a sum of orthogonal matrices.
	Hence \eqref{action-zs} follows from \eqref{CjA} and this orthogonality.
	\begin{figure}
		\centering
		\[
		\begin{tikzpicture}[decoration={brace,amplitude=5pt},baseline=(current bounding box.west)]
			\matrix (m) [
			matrix of math nodes,
			nodes in empty cells,
			nodes={minimum size=6.5mm, anchor=center},
			left delimiter=(,right delimiter=)
			] {
				1 &   &        &   &    &   &        & \\
				& 1 &        &   &    &   &        & \\
				&   & \rotatebox[origin=c]{-45}{$\cdots$} &   &    &   &        & \\
				&   &        & 1 &    &   &        & \\
				&   &        &   & \!\!\!\shortminus m\!\! &   &        & \\
				&   &        &   &    & 0 &        & 0\\
				&   &        &   &    &   & \rotatebox[origin=c]{-45}{$\cdots$} &\\
				&   &        &   &    &   &        & 0\\
			};
			\draw[decorate] (m-4-4.south west) -- (m-1-1.south west) node[below=5pt,midway,sloped] {\footnotesize $m$-many};
			\draw (m-1-1.north east) -- (m-1-1.south east) -- (m-2-2.north east) -- (m-2-2.south east) -- (m-3-3.north);
			\draw (m-3-3.east) -- (m-3-3.south east) -- (m-4-4.north east) -- (m-4-4.south east) -- (m-5-5.north east) -- (m-5-5.south east) -- (m-6-6.north east) -- (m-6-6.south east) -- (m-7-7.north);
			\draw (m-7-7.east) -- (m-7-7.south east) -- (m-8-8.north east)--(m-1-8.north east) -- (m-1-1.north east);
			\draw (m-4-4.north east) -- (m-1-4.north east);
			\draw (m-5-5.north east) -- (m-1-5.north east);
			\draw (m-5-5.north east) -- (m-5-8.north east);
			\draw (m-6-6.north east) -- (m-6-8.north east);
			\node[fit=(m-2-3)(m-2-4)]{$\renorm_m$};
			\node[fit=(m-1-6)(m-4-8), text height = 4em]{$\frac12\renorm_m$};
			\node[fit=(m-5-6)(m-5-8), text height = 1.17em]{$-\frac12m\renorm_m$};
			
			\draw (m-3-5.north) + (60:2.1) node(B){$\frac{1-m}{2}\renorm_m$};
			\path [out=90,in=195]   (m-2-5.north) edge (B);
		\end{tikzpicture}
		\]
		\caption{Collating the quantities $\tr[\bC_mA_{jk}^{(b)}]$ and  $\tr[\bC_mB_{jk}^{(b)}]$.
			In the upper triangle, a value $v$ in coordinate $(j,k)$ means $\tr[\bC_mA_{jk}^{(b)}]=\tr[\bC_mB_{jk}^{(b)}]=v$ for any $b$.
			For reference, the (unnormalized) definition of $\bC_m$ is recorded on the diagonal.}
		\label{fig:C-cancel-matrix}
	\end{figure}
\end{proof}

Now we are ready to prove Theorem \ref{thm:GM}.

\begin{proof}[Proof of Theorem \ref{thm:GM}]
Let us normalize $\rho$ as $r(x, y, z) :=\frac13\binom{\go}{2}^{-1} \rho (x, y, z)$, so $\tr [r(x, y, z)]=1$.
Now choose any $(\vec x, \vec y, \vec z) \in H_\go^n$ with 
\begin{equation*}
	\vec x=\big(x^{(1)},\dots, x^{(n)}\big), \qquad
	\vec y=\big(y^{(1)},\dots, y^{(n)}\big), \qquad
	\vec z=\big(z^{(1)},\dots, z^{(n)}\big), 
\end{equation*}
and $\big(x^{(j)},y^{(j)},z^{(j)}\big)\in H_\go$ for each $1\le j\le n$.
We can consider a tensor of $r$ states
\[
r(\vec x, \vec y, \vec z) := r\big(x^{(1)}, y^{(1)}, z^{(1)}\big) \otimes  r\big(x^{(2)}, y^{(2)}, z^{(2)}\big) \otimes\dots \otimes r\big(x^{(n)}, y^{(n)}, z^{(n)}\big)\,.
\]
Recall that any GM observable $\cA$ of degree at most $d$ has the unique expansion
\begin{equation*}
	\cA=\sum_{\al=(\al_1,\dots, \al_n)\in \Lambda_\go^n}\widehat{\cA}_\al M_{\al_1}\otimes \cdots \otimes M_{\al_n}
\end{equation*}
where $\{M_\al \}_{\al\in  \Lambda_\go}=\textnormal{GM}(\go)$ and $\widehat{\cA}_\al=0$ if more than $d$ matrices of $M_{\al_j}, 1\le j\le n$ are not identity matrices.

By Lemma \ref{action}, for any $\al=(\al_1,\dots, \al_n)\in \Lambda_\go^n$ with $|\alpha|=:\kappa\le d$, the map
\begin{equation*}
	(\vec x,\vec y,\vec z)\mapsto \tr \left[{\textstyle\bigotimes_{j=1}^n M_{\al_1}}\cdot r(\vec x, \vec y, \vec z)\right]
\end{equation*}
is a multilinear monomial of degree-$\kappa$ on the Boolean cube $H_\go^n = \{-1, 1\}^{n(\go^2-1)} $ with coefficient 
\[
\left(\frac{\sqrt{\go/2}}{3\binom{\go}{2}}\right)^\kappa.
\]

Since the coefficients of this scalar polynomial are of the form
\[
\left(\frac{\sqrt{\go/2}}{3\binom{\go}{2}}\right)^\kappa \widehat{\cA}_\al,\qquad 0\le \kappa \le d\,,
\]
the moduli of those coefficients are at least 
\[
\frac{1}{\big(\frac32(\go^2-\go)\big)^d} |\widehat{\cA}_\al|\,.
\]
So by hypercube Bohnenblust--Hille inequalities as in \cite{DMP} we have
\begin{equation*}
	\Big(\sum_\al |\widehat{\cA}_\al|^{\frac{2d}{d+1}}\Big)^{\frac{d+1}{2d}} \le \big(\tfrac32(\go^2-\go)\big)^d \textnormal{BH}^{\le d}_{\{\pm 1\}}\sup_{(\vec x,\vec y,\vec z)\in H_\go^n}|\tr(\cA\cdot r(\vec x, \vec y, \vec z)|\, .
\end{equation*}
On the other hand, because $r(x,y,z)$ is a density matrix,
\[
|\tr[\cA\cdot r(\vec x, \vec y, \vec z)]|\le \|\cA\|_{\textnormal{op}}\,.
\]
All combined, we get 
\begin{equation*}
\Big(\sum_\al |\widehat{\cA}_\al|^{\frac{2d}{d+1}}\Big)^{\frac{d+1}{2d}} \le \big(\tfrac32(\go^2-\go)\big)^d C^{\sqrt{d\log d}} \|\cA\|_{\textnormal{op}}\,.
\qedhere
\end{equation*}
\end{proof}

\subsection{Qudit Bohnenblust--Hille in the Heisenberg--Weyl basis}
\label{sec:hw-bh}

Here we give a proof of Theorem \ref{thm:bh HW} by reduction to the Cyclic BH inequality.
We collect first a few facts about $X$ and $Z$.
In what follows $\langle\cdot,\cdot\rangle$ denotes the inner product on $\C^n$ that is linear in the second argument.

\begin{lemma}\label{lem:facts}
	We have the following:
	\begin{enumerate}
		\item $\{X^\ell Z^m: \ell, m\in \mathbb{Z}_\go\}$ form a basis of $M_\go(\C)$.
		\item For all $k,\ell,m\in \mathbb{Z}_\go$:
		\begin{equation*}
			(X^\ell Z^m)^k=\om^{\frac{1}{2}k(k-1)\ell m}X^{k\ell} Z^{km}
		\end{equation*}
		and for all $\ell_1,\ell_2,m_1,m_2\in\mathbb{Z}_\go$:
		\begin{equation*}
			X^{\ell_1}Z^{m_1} X^{\ell_2}Z^{m_2}=\om^{\ell_2 m_1-\ell_1 m_2}X^{\ell_2}Z^{m_2}X^{\ell_1}Z^{m_1}.
		\end{equation*}
		\item If $\go$ is prime, then for any $(0,0)\neq (\ell,m)\in \mathbb{Z}_\go\times \mathbb{Z}_\go$, the eigenvalues of $X^\ell Z^m$ are $\{1,\om,\dots, \om^{\go-1}\}$. This is not the case if $\go$ is not prime. 
	\end{enumerate}
\end{lemma}

\begin{proof}Considering each statement one-by-one:
	\begin{enumerate}
		\item Suppose that $\sum_{\ell,m}a_{\ell,m}X^\ell Z^m=0$. For any $j,k\in \mathbb{Z}_\go$, we have
		\begin{equation*}
			\sum_{\ell,m} a_{\ell,m}	\langle X^\ell Z^m e_j,e_{j+k}\rangle 
			=\sum_{m} a_{k,m}	\om^{jm}
			=0.
		\end{equation*}
		Since the Vandermonde matrix associated to $(1,\om,\dots,\om^{\go-1})$ is invertible, we have $a_{k,m}=0$ for all $k,m\in \mathbb{Z}_\go$.
		\item It follows immediately from the identity $ZX=\om XZ$ which can be verified directly: for all $j\in \mathbb{Z}_\go$
		\begin{equation*}
			ZXe_j= Ze_{j+1}=\om^{j+1}e_{j+1}=\om^{j+1} X e_{j}=\om XZ e_{j}.
		\end{equation*}
		\item Assume $\go$ to be prime and $(\ell,m)\neq (0,0)$. If $\ell=0$ and $m\neq 0$, then the eigenvalues of $Z^m$ are
		\[
    \{\om^{jm}:j\in\mathbb{Z}_\go\}=\{\om^{j}:j\in \mathbb{Z}_\go\},
    \]
		since $\go$ is prime. If $\ell\neq 0$, then we may relabel the standard basis $\{e_j:j\in\mathbb{Z}_\go\}$ as $\{e_{j\ell}:j\in\mathbb{Z}_\go\}$. Consider the non-zero vectors 
		\begin{equation*}
			\zeta_k:=\sum_{j\in \mathbb{Z}_\go}\om^{\frac{1}{2}j(j-1)\ell m-jk}e_{j\ell},\qquad k\in \mathbb{Z}_\go.
		\end{equation*}
		A direct computation shows: for all $k\in \mathbb{Z}_\go$
		\begin{align*}
			X^{\ell}Z^m \zeta_k
			&=\sum_{j\in \mathbb{Z}_\go}\om^{\frac{1}{2}j(j-1)\ell m-jk}\cdot \om^{j\ell m}X^{\ell}e_{j\ell}\\
			&=\sum_{j\in \mathbb{Z}_\go}\om^{\frac{1}{2}j(j+1)\ell m-jk}e_{(j+1)\ell}\\
			&=\sum_{j\in \mathbb{Z}_\go}\om^{\frac{1}{2}j(j-1)\ell m-jk+k}e_{j\ell}\\
			&=\om^k\zeta_k.
		\end{align*}
		If $\go$ is not prime, say $\go=\go_1 \go_2$ with $\go_1,\go_2>1$, then $X^{\go_1}$ has $1$ as eigenvalue with multiplicity $\go_1>1$. So we do need $\go$ to be prime.
        \qedhere
	\end{enumerate}
\end{proof}

Let us record the following observation as a lemma. 

\begin{lemma}\label{lem:orthogonal}
	Suppose that $k\ge 1$, $A,B$ are two unitary matrices such that $B^k=\un$, $AB=\lambda BA$ with $\lambda\in\C$ and $\lambda\ne 1$. 
	Suppose that $\xi$ is a non-zero vector such that $B\xi =\mu \xi$ ($\mu\ne 0$ since $\mu^k=1$). Then 
	\begin{equation*}
		\langle \xi,A\xi\rangle=0.
	\end{equation*}
\end{lemma}

\begin{proof}
	By assumption
	\[
	\mu\langle \xi,A\xi\rangle
	=\langle \xi,AB\xi\rangle 
	=\lambda \langle \xi,BA\xi\rangle.
	\]
	Since $B^\dagger=B^{k-1}$, $B^\dagger\xi=B^{k-1}\xi=\mu^{k-1}\xi=\overline{\mu}\xi$. Thus
	\[
	\mu\langle \xi,A\xi\rangle
	=\lambda \langle \xi,BA\xi\rangle
	=\lambda \langle B^\dagger\xi,A\xi\rangle
	=\lambda \mu\langle \xi,A\xi\rangle.
	\]
	Hence, $\mu(\lambda-1)\langle \xi,A\xi\rangle=0$. This gives $\langle \xi,A\xi\rangle=0$ as $\mu(\lambda-1)\ne 0$. 
\end{proof}

Now we are ready to prove Theorem \ref{thm:bh HW}:
\begin{proof}[Proof of Theorem \ref{thm:bh HW}]
Fix a prime number $\go\ge 2$. Recall that $\om=e^{\frac{2\pi i}{\go}}$. Consider the generator set of $\mathbb{Z}_\go\times \mathbb{Z}_\go$
\[
\Sigma_\go:=\{(1,0),(1,1),\dots, (1,\go-1),(0,1)\}.
\]
For any $z\in\Om_\go$ and $(\ell,m)\in \Sigma_\go$, we denote by $e^{\ell,m}_{z}$ the unit eigenvector of $X^\ell Z^m$ corresponding to the eigenvalue $z$.
For any vector $\vec{\om}\in(\Om_\go)^{(\go+1)n}$ of the form 
\begin{equation}\label{eq:defn of vec omega}
	\vec{\om}=({\vec{\om}}^{\ell,m})_{(\ell,m)\in \Sigma_\go},
	\qquad {\vec{\om}}^{\ell,m}=(\omega^{\ell,m}_1,\dots, \omega^{\ell,m}_n)\in (\Om_\go)^{(\go+1)n},
\end{equation}
we consider the matrix 
\[
\rho(\vec{\om}):=\rho_{1}(\vec{\om})\otimes \cdots\otimes \rho_{n}(\vec{\om})
\]
where 
\[
\rho_k(\vec{\om}):=\frac{1}{\go+1}\sum_{(\ell,m)\in \Sigma_\go}\ketbra{e^{\ell,m}_{\om^{\ell,m}_k}}{e^{\ell,m}_{\om^{\ell,m}_k}}.
\]
Then each $\rho_{k}(\vec{\om})$ is a density matrix and so is $\rho(\vec{\om})$. 

Suppose that $(\ell,m)\in \Sigma_\go$ and $(\ell',m')\notin\{(k\ell,km):(\ell,m)\in \Sigma_\go\}$, then by Lemma \ref{lem:facts}
\begin{equation*}
	X^{\ell'}Z^{m'}X^{\ell}Z^{m}=\om^{\ell m'-\ell'm}X^{\ell}Z^{m}X^{\ell'}Z^{m'}.
\end{equation*}
From our choice $\om^{\ell m'-\ell'm}\neq 1$. By Lemmas \ref{lem:facts} and \ref{lem:orthogonal}
\begin{equation*}
	\tr[X^{\ell'}Z^{m'}\ketbra{e^{\ell,m}_{z}}{e^{\ell,m}_{z}}]
	=\langle X^{\ell'}Z^{m'}e^{\ell,m}_{z},e^{\ell,m}_{z}\rangle=0, \qquad z\in \Om_\go.
\end{equation*}
Suppose that $(\ell,m)\in \Sigma_\go$ and $1\le k\le \go-1$. Then by Lemma \ref{lem:facts}
\begin{align*}
	\tr[X^{k\ell}Z^{k m}\ketbra{e^{\ell,m}_{z}}{e^{\ell,m}_{z}}]
	&=\om^{-\frac{1}{2}k(k-1)\ell m}\langle (X^\ell Z^m)^k e^{\ell,m}_{z},e^{\ell,m}_{z}\rangle\\
	&=\om^{-\frac{1}{2}k(k-1)\ell m}z^k, \qquad z\in \Om_\go.
\end{align*}
All combined, for all $1\le k\le \go-1,(\ell,m)\in \Sigma_\go$ and $1\le i\le n$ we get
\begin{align*}
	\tr[X^{k\ell} Z^{k m } \rho_i(\vec{\om})]
	&=\frac{1}{\go+1}\sum_{(\ell',m')\in \Sigma_\go}\langle e^{\ell',m'}_{\om^{\ell',m'}_i}, X^{k\ell}Z^{ km}e^{\ell',m'}_{\om^{\ell',m'}_i}\rangle\\
	&=\frac{1}{\go+1}\langle e^{\ell,m}_{\om^{\ell,m}_i}, X^{k\ell}Z^{ km}e^{\ell,m}_{\om^{\ell,m}_i}\rangle\\
	&=\frac{1}{\go+1}\om^{-\frac{1}{2}k(k-1)\ell m}(\om^{\ell,m}_i)^k.
\end{align*}

Recall that any degree-$d$ polynomial in $M_\go(\C)^{\otimes n}$ is a linear combination of monomials
\[
A(\vec{k},\vec{\ell},\vec{m};\vec{i}):=
\cdots \otimes X^{k_1\ell_1}Z^{k_1 m_1}\otimes \cdots \otimes X^{k_\kappa\ell_\kappa}Z^{k_\kappa m_\kappa}\otimes\cdots
\]
where
\begin{itemize}
	\item $\vec{k}=(k_1,\dots,k_\kappa)\in \{1,\dots, \go-1\}^\kappa$ with $0\le \sum_{j=1}^{\kappa}k_j\le d$;
	\item $\vec{\ell}=(\ell_1,\dots, \ell_\kappa),\vec{m}=(m_1,\dots, m_\kappa)$ with each $(\ell_j,m_j)\in \Sigma_\go$;
	\item $\vec{i}=(i_1,\dots, i_\kappa)$ with $1\le i_1 <\cdots<i_\kappa\le n$;
	\item $X^{k_j\ell_j}Z^{k_j m_j}$ appears in the $i_j$-th place, $1\le j\le \kappa$, and all the other $n-\kappa$ elements in the tensor product are the identity matrices $\un$.
\end{itemize}
So for any $\vec{\om}\in (\Om_\go)^{(\go+1)n}$ of the form \eqref{eq:defn of vec omega} we have from the above discussion that
\begin{align*}
	\tr[A(\vec{k},\vec{\ell},\vec{m};\vec{i})\rho(\vec{\om})]
	&=\prod_{j=1}^{\kappa}\tr[X^{k_j\ell_j}Z^{k_j m_j}\rho_{i_j}(\vec{\om})]\\
	&=\frac{\om^{-\frac{1}{2}\sum_{j=1}^{\kappa}k_j(k_j-1)\ell_j m_j}}{(\go+1)^{\kappa}}(\om^{\ell_1,m_1}_{i_1})^{k_1}\cdots (\om^{\ell_\kappa,m_\kappa}_{i_\kappa})^{k_\kappa}.
\end{align*}
So $\vec{\om}\mapsto 	\tr[A(\vec{k},\vec{\ell},\vec{m};\vec{i})\rho(\vec{\om})]$ is a polynomial on $(\Om_\go)^{(\go+1)n}$ of degree at most $\sum_{j=1}^{\kappa}k_j\le d$. 

Now for general polynomial $A\in M_\go(\C)^{\otimes n}$ of degree $d$:
\begin{equation*}
	A=\sum_{\vec{k},\vec{\ell},\vec{m},\vec{i}} c(\vec{k},\vec{\ell},\vec{m};\vec{i})A(\vec{k},\vec{\ell},\vec{m};\vec{i})
\end{equation*}
where the sum runs over the above $(\vec{k},\vec{\ell},\vec{m};\vec{i})$. This is the Fourier expansion of $A$ and each $c(\vec{k},\vec{\ell},\vec{m};\vec{i})\in \C$ is the Fourier coefficient. So 
\begin{equation*}
	\|\widehat{A}\|_p=\left(\sum_{\vec{k},\vec{\ell},\vec{m},\vec{i}}|c(\vec{k},\vec{\ell},\vec{m};\vec{i})|^p\right)^{1/p}.
\end{equation*}
To each $A$ we assign the function $f_A$ on $(\Om_\go)^{(\go+1)n}$ given by
\begin{align*}
	f_A(\vec{\om})&=\tr[A\rho(\vec{\om})]\\
	&=\sum_{\vec{k},\vec{\ell},\vec{m},\vec{i}} \frac{\om^{-\frac{1}{2}\sum_{j=1}^{\kappa}k_j(k_j-1)\ell_j m_j}c(\vec{k},\vec{\ell},\vec{m};\vec{i})}{(\go+1)^{\kappa}}(\om^{\ell_1,m_1}_{i_1})^{k_1}\cdots (\om^{\ell_\kappa,m_\kappa}_{i_\kappa})^{k_\kappa}.
\end{align*}
Note that this is the Fourier expansion of $f_A$ since the monomials $(\om^{\ell_1,m_1}_{i_1})^{k_1}\cdots (\om^{\ell_\kappa,m_\kappa}_{i_\kappa})^{k_\kappa}$ differ for different $(\vec{k},\vec{\ell},\vec{m},\vec{i})$. Therefore,
\begin{align*}
	\|\widehat{f_A}\|_p
	&=\left(\sum_{\vec{k},\vec{\ell},\vec{m},\vec{i}}\left|\frac{c(\vec{k},\vec{\ell},\vec{m};\vec{i})}{(\go+1)^{\kappa}}\right|^p\right)^{1/p}\\
	&\ge \frac{1}{(\go+1)^{d}}\left(\sum_{\vec{k},\vec{\ell},\vec{m},\vec{i}}|c(\vec{k},\vec{\ell},\vec{m};\vec{i})|^p\right)^{1/p}\\
	&=\frac{1}{(\go+1)^{d}}	\|\widehat{A}\|_p.
\end{align*}
Using the Bohnenblust--Hille inequalities for cyclic groups (Corollary \ref{cor:bh cyclic groups}), we have
\begin{equation*}
	\|\widehat{f_A}\|_{\frac{2d}{d+1}}\le C(d)\|f_A\|_{\Om_\go^{(\go+1)n}}
\end{equation*}
for some $C(d)>0$. All combined, we obtain
\begin{equation*}
	\|\widehat{A}\|_{\frac{2d}{d+1}}
	\le (\go+1)^d	\|\widehat{f_A}\|_{\frac{2d}{d+1}}
	\le (\go+1)^dC(d)\|f_A\|_{(\Om_\go)^{(\go+1)n}}
	\le (\go+1)^dC(d)\|A\|_{\textnormal{op}}\,. \qedhere
\end{equation*}
\end{proof}

\section{Applications to learning}
\label{sec:learning}
We now apply our new BH inequalities to obtain learning results for functions on $\Z_\go^n$ and operators on qudits.
In the latter case, as for qubits \cite{CHP}, this result may be extended to quantum observables of arbitrary complexity.

Fourier sampling is approached differently in the cyclic and qudit contexts, so we isolate the Eskenazis--Ivanisvili approximation principle from the Boolean function learning aspects of \cite{EI22} and include a proof for completeness.
We'll also need it for vectors in $\C$ rather than $\R$ as it appeared originally but the proof is essentially identical.

\begin{theorem}[Generic Eskenazis--Ivanisvili]
\label{thm:generic-EI}
Let $d\in \mathbb{N}$ and $\eta, B> 0$.
Suppose $v,w\in \C^n$ with $\|v-w\|_\infty\leq \eta$ and $\|v\|_{\frac{2d}{d+1}}\leq B$.
Then for $\widetilde{w}$ defined as $\widetilde{w}_j= w_j \mathbbm{1}_{[|w_j| \geq \eta(1+\sqrt{d+1})]}$
we have the bound
\[\|\widetilde{w}-v\|_{2}^2\leq (e^5\eta^2 d B^{2d})^{\frac{1}{d+1}}.\]
\end{theorem}

\begin{proof}
Let $t>0$ be a threshold parameter to be chosen later.
Define $S_t = \{j: |w_j|\geq t\}$ and note from the triangle inequality in $\C$ that
\begin{align}
    |v_j|\geq |w_j|-|v_j-w_j| = t-\eta & \text{ for } j\in S_t
    \label{eq:big-coords}\\
    |v_j|\leq |w_j|+|v_j-w_j| = t+\eta & \text{ for } j\not\in S_t.
    \label{eq:small-coords}
\end{align}
We may also estimate $|S_t|$ as
\begin{equation}
\label{eq:big-coords-cardinality}
|S_t| = \sum_{j\in S}\frac{|v_j|}{|v_j|}\overset{\eqref{eq:big-coords}}{\leq} (t-\eta)^{-\frac{2d}{d+1}}\sum_{j\in[n]}|v_j|^{\frac{2d}{d+1}}\leq (t-\eta)^{-\frac{2d}{d+1}}\|v\|_{\frac{2d}{d+1}}^{\frac{2d}{d+1}}\leq (t-\eta)^{-\frac{2d}{d+1}}B^{\frac{2d}{d+1}}.
\end{equation}
With $\widetilde{w}^{(t)} := (w_j\mathbbm{1}_{[w_j\geq t]})_{j=1}^n$, we find
\begin{align*}
\|\widetilde{w}^{(t)}-v\|_2^2&=\sum_{j\in S_t} |w_j-v_j|^2 + \sum_{j\not\in S_t}|v_j|^2 \overset{\eqref{eq:small-coords}}{\leq} |S_t|\eta^2 + (t+\eta)^{\frac{2}{d+1}}\sum_{j\in[n]}|v_j|^{\frac{2d}{d+1}}\\
&\overset{\eqref{eq:big-coords-cardinality}}{\leq} B^{\frac{2d}{d+1}}\left(\eta^2(t-\eta)^{-\frac{2d}{d+1}}+(t+\eta)^{\frac{2}{d+1}}\right).
\end{align*}
Choosing $t=\eta(1+\sqrt{d+1})$ then yields $\widetilde{w}$ with error, after some careful scalar estimates,
\[
\|\widetilde{w}-v\|_2^2\leq (e^5\eta^2 d B^{2d})^{\frac{1}{d+1}}.
\]
See \cite[Eqs. 18 \& 19]{EI22} for details on the scalar estimates.
\end{proof}

In the context of low-degree learning, $v$ is the true vector of Fourier coefficients, and $w$ is the vector of empirical coefficients obtained through Fourier sampling.

\subsection{Cyclic group learning}
\begin{theorem}
    Let $f:\Z_\go^n\to \mathbb{D}$ be a degree-$d$ function.
    Then with $(\log \go)^{\mathcal{O}(d^2)}\log(n/\delta)\eps^{-d-1}$ independent random samples $(x,f(x))$, $x\sim\mathcal{U}(\Z_\go^n)$, we may with confidence $1-\delta$ learn a function $\widetilde{f}:\Z_\go^n\to\C$ with $\|f-\widetilde{f}\|_2^2\leq \eps$.
\end{theorem}

\begin{proof}
Let $f=\sum_{\al}\widehat{f}(\al)z^\al$ be the Fourier expansion. For a number of samples $s$ to be specified later, sample $x^{(1)},\ldots, x^{(s)}\overset{\mathsf{iid}}{\sim}\mathcal{U}(\{0,1,\ldots,\go-1\}^n)$ and for each $\alpha\in\Z_\go^n$ with $|\alpha|\leq d$, form the empirical Fourier coefficient
\[
w_\alpha:=\frac{1}{s}\sum_{j=1}^{s}f(x^{(j)})\omega_\go^{-\sum_{\ell=1}^n\alpha_\ell x^{(j)}_\ell},
\]
where $\omega_\go=e^{\frac{2\pi i}{\go}}$ and $x^{(j)}=(x^{(j)}_1,\dots, x^{(j)}_n)$.
Then ${w}_\alpha$ is a sum of bounded i.i.d. random variables with expected value $\widehat{f}(\alpha)$, so Chernoff gives
\begin{align*}
\Pr[|\widehat{f}(\alpha)-w_\alpha|\geq \eta]&=\Pr\big[\Re\big(\widehat{f}(\alpha)- w_\alpha\big)^2+\Im\big(\widehat{f}(\alpha)-w_\alpha\big)^2\geq \eta^2\big]\\
&\leq \Pr\big[|\Re\big(\widehat{f}(\alpha)-w_\alpha\big)|\geq \eta/\sqrt2\big]+\Pr\big[|\Im\big(\widehat{f}(\alpha)-w_\alpha\big)|\geq \eta/\sqrt2\big]\\
&\leq 4\exp\left(-s\eta^2/4\right)\,.
\end{align*}
So the probability we simultaneously estimate all nonzero Fourier coefficients of $f$ to within $\eta$ is
\[\Pr\big[|\widehat{f}(\alpha)-w_\alpha|< \eta\text{ for all } \alpha \text{ with } |\alpha|\leq d\big]\geq 1- 4\sum_{k=0}^d\binom{n}{k}\exp\left(\frac{-s\eta^2}{4}\right)\,,\]
which in turn we will require to be $\geq 1-\delta$.

Now applying Theorem \ref{thm:generic-EI} to obtain $\widetilde{w}$ and recalling $\|\widehat f\|_{\frac{2d}{d+1}}\leq \BH_{\Z_\go}^{\le d}\|f\|_\infty=\BH^{\le d}_{\Z_\go}$ we have that with probability $1-\delta$, the function $\widetilde{f}(x) := \sum_{\alpha}\widetilde{w}_\alpha \prod_{j=1}^n\omega_\go^{\alpha_jx_j}$ has $L_2$ error
\begin{equation}
\label{eq:cyclic-err}
\|\widetilde{f}-f\|_2^2 \overset{\text{(Parseval)}}{=} \sum_{\alpha}|\widehat{f}(\alpha)-\widetilde{\omega}_\alpha|^2\leq \big(e^5\eta^2d(\BH_{\Z_\go}^{\le d})^{2d}\big)^{\frac{1}{d+1}}\,.
\end{equation}
So in order to achieve $\|\widetilde{f}-f\|_2^2\leq \eps$ it is enough to pick $\eta^2=\eps^{d+1}e^{-5}d^{-1}(\BH_{\Z_\go}^{\le d})^{-2d}$, which entails by standard estimates that taking a number of samples $s$ with
\[s\geq\frac{4e^5d^2(\BH_{\Z_\go}^{\le d})^{2d}}{\eps^{d+1}}\log\left(\frac{4en}{\delta}\right)\]
suffices.
\end{proof}

\subsection{Qudit learning}
We first pursue a learning algorithm that finds a (normalized) $L_2$ approximation to a low-degree operator $\cA$.
Then we'll see how this extends to an algorithm finding an approximation $\widetilde{\cA}$ with good mean-squared error over certain distributions of states for target operators $\cA$ of any degree.

We make a couple assumptions for clarity and brevity.
First, we assume the unknown observable $\cA$ has operator norm $\|\cA\|_{\textnormal{op}}\leq 1$.
We will also assume that for a mixed state $\rho$ the quantity $\tr[\cA\rho]$ can be directly computed.
Of course this is not true in practice; in the lab one must take many copies of $\rho$, collect observations $m_1,\ldots, m_s$ and form the estimate $\frac{1}{s}\sum_jm_j\approx \tr[\cA\rho]$.
The analysis required to relax these assumptions from the following results are routine so we omit them.

\subsubsection{Low-degree Qudit learning}
\begin{theorem}[Low-degree Qudit Learning]
    \label{thm:qudit-ld-learning}
    Let $\cA$ be a degree-$d$ observable on $n$ qudits with $\|\cA\|_\textnormal{op}\leq 1$.
    Then there is a collection $S$ of product states such that with a number
    \[\mathcal{O}\Big(\big( \go\|\cA\|_\mathrm{op}\big)^{C\cdot d^2}d^2\eps^{-d-1}\log\left(\tfrac{n}{\delta}\right)\Big)\] of samples of the form $(\rho, \tr[\cA \rho]),$ $\rho\sim\mathcal{U}(S)$, we may with confidence $1-\delta$ learn an observable $\widetilde{\cA}$ with $\|\cA-\widetilde{\cA}\|_2^2\leq \eps$.
\end{theorem}
Here $\|\cA\|_2$ denotes the normalized $L_2$ norm induced by the inner product $\langle A,B\rangle:=\go^{-n}\tr[A^\dagger B]$.
Also, we choose to include explicit mention of $\|\cA\|_\mathrm{op}$ here as it will be useful later.
For applications it is natural to assume $\|\cA\|_\mathrm{op}$ is bounded independent of $n$.

\begin{proof}
    We will first pursue an $L_\infty$ estimate of the Fourier coefficients in the Gell-Mann basis.
    To that end, sample $\vec{\boldsymbol{x}}_1,\ldots, \vec{\boldsymbol{x}}_s\overset{\mathsf{iid}}{\sim} \{-1,1\}^{n(\go^2-1)}$.
    As in the proof of Theorem \ref{thm:GM}, for any such $\vec{\boldsymbol x}$ we partition indices as $\vec{\boldsymbol{x}} = (\boldsymbol{x}_1,\ldots, \boldsymbol{x}_n)\in(\{-1,1\}^{\go^2-1})^n$ with each $\boldsymbol{x}_\ell$, $1\leq\ell\leq n$, corresponding to a qudit.
    Each $\boldsymbol{x}_\ell$ is further partitioned as
    \[\boldsymbol{x}_\ell=(x^{(\ell)},y^{(\ell)},z^{(\ell)})\in\{-1,1\}^{\binom{\go}{2}}\times\{-1,1\}^{\binom{\go}{2}}\times\{-1,1\}^{\go-1}\,,\]
    with each sub-coordinate associated with a specific Gell-Mann basis element for that qudit.

    Again for each $\vec{\boldsymbol{x}}$, for each qudit $\ell\in[n]$ form the mixed state
    \[r(x^{(\ell)},y^{(\ell)},z^{(\ell)})=\frac{1}{3\binom{\go}{2}}\left(\sum_{1\le j<k\le \go}  A_{jk}^{(x^{(\ell)}_{jk})} + \sum_{1\le j<k\le \go} B_{jk}^{(y^{(\ell)}_{jk})} + \sum_{m=1}^{\go-1} z^{(\ell)}_m \tfrac{1}{\sqrt{2\go}}\bC_m+ \tfrac{\go-1}{2}\cdot  \un\right).\]
    Then we may define for $\vec{\boldsymbol{x}}$ the $n$ qudit mixed state
    \[r(\vec{\boldsymbol{x}}) = \bigotimes_{\ell=1}^{n}r(x^{(\ell)},y^{(\ell)},z^{(\ell)})\]
    and consider the function
    \[f_\cA(\vec{\boldsymbol{x}}):=\tr[\cA \cdot r(\vec{\boldsymbol{x}})].\]
    Let $S(\alpha)$ denote the index map from the GM basis to subsets of $[n(\go^2-1)]$.
    With these states in hand and in view of the identity
    \[\widehat{f_{\cA}}\big(S(\alpha)\big)=c^{|\alpha|}\widehat{\cA}(\alpha) \quad \text{with} \quad c:= \frac{\sqrt{\go/2}}{3\binom{\go}{2}}\;<1,\] we may now define the empirical Fourier coefficients
    \[\cW(\alpha) = c^{-|\alpha|}\cdot\frac{1}{s}\sum_{t=1}^sf_{\cA}(\vec{\boldsymbol{x}}_t)\prod_{j\in S(\alpha)}x_j \;=\; c^{-|\alpha|}\frac{1}{s}\sum_{t=1}^s\tr[\cA\cdot r(\vec{x})]\prod_{j\in S(\alpha)}x_j.\]
    The coefficient $\cW(\alpha)$ is a sum of bounded i.i.d. random variables each with expectation $\widehat{\cA}(\alpha)$, so by Chernoff we have
    \[\Pr\big[|\cW(\alpha)-\widehat{\cA}(\alpha)|\geq \eta\big]\leq 2\exp(-s\eta^2c^{|\alpha|})\,.\]
    Taking the union bound, we find as before the chance of achieving $\ell_\infty$ error $\eta$ is:
    \[\Pr\big[|\cW(\alpha)-\widehat{\cA}(\alpha)|<\eta\text{ for all $\alpha$ with } |\alpha|\leq d\big] \geq 1-2\sum_{k=0}^d\binom{n}{k}\exp(-s\eta^2c^{d})\,,\]
    which again we shall require to be $\geq 1-\delta$.

    Applying Theorem \ref{thm:generic-EI} to obtain $\widetilde{\cW}$ and recalling $\|\widehat{\cA}\|_{\frac{2d}{d+1}}\leq \BH^{\leq d}_{\mathrm{GM}(\go)}\|\cA\|_\text{op}$ we find the estimated operator
    \[\widetilde{\cA}:=\sum_\alpha\widetilde{\cW}(\alpha)M_\alpha\]
    has $L_2$-squared error
    \[\|\widetilde{\cA}-\cA\|_2^2\overset{\text{(Parseval)}}{=}\sum_\alpha\big|\widetilde{\cW}(\alpha)-\widehat{\cA}(\alpha)\big|^2\leq \big(e^5\eta^2d(\BH^{\leq d}_{\mathrm{GM}(\go)}\|\cA\|_\text{op})^{2d}\big)^{\frac{1}{d+1}}\,.\]
    Thus to obtain error $\leq \eps$ it suffices to pick $\eta^2= \eps^{d+1}e^{-5}d^{-1}(\BH^{\leq d}_{\mathrm{GM}(\go)}\|\cA\|_\text{op})^{-2d}$, which entails by standard estimates that the algorithm will meet the requirements with a sample count of
    \[s\geq e^6\go^{3/2}d^2(\BH_{\mathrm{GM}(\go)}^{\leq d}\|\cA\|_\text{op})^{2d}\log\left(\tfrac{2en}{\delta}\right)\eps^{-d-1}\,. \qedhere\]
\end{proof}

\subsubsection{Learning arbitrary qudit observables}
As observed by Huang, Chen, and Preskill in \cite{CHP}, there are certain distributions $\mu$ of input states $\rho$ for which a low-degree truncation $\widetilde\cA$ of any observable $\cA$ gives a suitable approximation as measured by $\E_{\rho\sim \mu}|\tr[\widetilde\cA\rho]-\tr[\cA\rho]|^2$.
This observation extends easily to qudits, which in turn ends up generalizing the phenomenon in the context of qubits as well.
\label{sec:learning-arb}
\begin{definition}
    For a 1-qudit unitary $U$ let $U_j:=\mathbf{I}^{\otimes j-1}\otimes U\otimes \mathbf{I}^{\otimes n-j}$.
    Then for a probability distribution $\mu$ on $n$-qudit densities and $j\in[n]$ let $\mathop{\mathrm{Stab}}_j(\mu)$ be the set of $U\in \mathrm{U}(\go)$ such that for all densities $\rho$,
    \[\mu(\rho) = \mu(U_j\rho U_j^{\dagger})\,.\]
\end{definition}
\begin{remark}
    For a set $S$ of states, define $\mathop{\mathrm{Stab}}_j(S)=\{U\in\mathrm{U}(\go):U_jSU_j^\dagger\subseteq S\}$.
    Then it can be seen easily that $\mathop{\mathrm{Stab}}_j(\mu)$ is equal to the intersection of the stabilizers of the level sets of $\mu$.
    That is, $\mathop{\mathrm{Stab}}_j(\mu)=\bigcap_{0\leq r\leq 1}\mathop{\mathrm{Stab}}_j\big(\mu^{-1}(r)\big)$.
\end{remark}

Recall the definition of a unitary $t$-design.

\begin{definition}
    Consider $P_{t,t}(U)$, a polynomial of degree at most $t$ in the matrix elements of unitary $U\in\mathrm{U}(\go)$ and of degree at most $t$ in the matrix elements of $U^\dagger$.
    Then a finite subset $S$ of the unitary group $\mathrm{U}(\go)$ is a \emph{unitary t-design} if for all such $P_{t,t}$,
    \[\frac{1}{|S|}\sum_{U\in S}P_{t,t}(U)=\E_{U\sim\mathrm{Haar}(\mathrm{U}(\go))}[P_{t,t}(U)]\,.\]
\end{definition}

We are ready to name the distributions for which low-degree truncation is possible without losing much accuracy.
\begin{definition}
    Call a distribution $\mu$ on $n$-qudit densities \emph{locally 2-design invariant} (L2DI) if for all $j\in[n]$, $\mathrm{Stab}_j(\mu)$ contains a unitary $2$-design.
\end{definition}

Of course, the $n$-fold tensor product of Haar-random qudits is an L2DI distribution, but there are many other possible distributions and in general they can be highly entangled.
When $\go=2$ and the $2$-design leaving $\mu$ locally invariant is the single-qubit Clifford group, such distributions are termed \emph{locally flat} in \cite{CHP}.
For any prime $\go$ the Clifford group on $\mathcal{H}_\go$ is a 2-design \cite{GAE}.
Importantly, however, any distribution just on classical inputs $\ket{x}, x\in \{0,1,\ldots, \go-1\}^n$ is not L2DI, as a consequence of the following general observation:

\begin{proposition}
    Suppose $\mu$ is an L2DI distribution over pure product states $\rho = \bigotimes_{j=1}^n\ketbra{\psi_j}{\psi_j}$.
    For $j\in[n]$ define $S_j=\{\rho_{\{j\}}:\rho\in\mathrm{supp}(\mu)\}$.
    Then $|S_j|\geq \go^2$.
\end{proposition}
\begin{proof}
    By the L2DI property we have a 2-design $X\subseteq\mathrm{SU}(\go)$ under which $\mu$ is locally invariant.
    Let $\rho$ be a pure state in $S_j$ and note that $P:=\{U\rho U^\dagger: U\in X\}\subseteq S_j$.
    On the other hand, $P$ forms a complex-projective 2-design \cite{dankert}.
    And any complex-projective 2-design in $\go$ dimensions must have cardinality at least $\go^2$ \cite[Theorem 4]{Scott_2006}.
\end{proof}


The truncation theorem for qudits goes through much the same way as it does for locally flat qubit distributions in \cite{CHP}.
\begin{theorem}
    \label{thm:L2DI-NS}
    Let $\cA$ be an operator on $\mathcal{H}_\go^{\otimes n}$ and $\mu$ a probability distribution on densities.
    Then if $\mu$ is L2DI we have
    \[\E_{\rho\sim\mu}|\tr[\cA\rho]|^2\leq \sum_\alpha\big(\tfrac{\go}{\go^2-1}\big)^{|\alpha|}|\widehat{\cA}(\alpha)|^2\,.\]
\end{theorem}
The reader will notice a marked similarity to the Fourier-basis expression of noise stability $\E_{x\sim_\delta y}f(x)f(y)=\langle f, \mathrm{T}\hspace{-1pt}_\delta f\rangle = \sum_{S\subseteq[n]}\delta^{|S|}\widehat{f}(S)^2$ for Boolean functions (\emph{e.g.}, \cite{Odonnell}).
\begin{proof}
    Expanding $\cA$ in the Gell-Mann basis we obtain
    \begin{align}
    \E_{\rho\sim\mu}|\tr[\cA\rho]|^2 \;=\; \E_\rho \Big|\sum_{\alpha}\widehat{\cA}(\alpha)\tr[M_\alpha\cdot\rho]\Big|^2
    \;=\; \sum_{\alpha,\beta}\overline{\widehat{\cA}(\alpha)}\widehat{\cA}(\beta)\underbrace{\E_\rho\tr[M_\alpha\otimes M_\beta \cdot \rho\otimes\rho]}_{(*)}\,.
    \label{eq:difference-decomp}
    \end{align}
    
    Now for fixed $\alpha,\beta$ we examine the expectation $(*)$.
    For $j\in[n]$ let $D_j$ be a $2$-design in $\mathrm{Stab}_j(\mu)$ guaranteed to exist by the L2DI property and define $D=\bigotimes_{j=1}^nD_j$ in the natural way.
    Then because $D_j$ leaves $\mu$ invariant,
    \begin{align*}
    (*)=\E_\rho\E_{U\sim\mathcal{U}(D)}\tr[M_\alpha\otimes M_\beta \cdot U\rho U^\dagger\otimes U\rho U^\dagger]
    =\E_\rho\tr\left[\bigotimes_{j=1}^n\E_{U_j\sim \mathcal{U}(D_j)}(U_j^\dagger M_{\alpha_j}U_j)\otimes(U_j^\dagger M_{\beta_j}U_j) \cdot \rho\otimes\rho\right].
    \end{align*}

    For a single coordinate $j$ let us study $(U_j^\dagger M_{\alpha_j}U_j)\otimes(U_j^\dagger M_{\beta_j}U_j)$.
    Certainly if $M_{\alpha_j}=M_{\beta_j}=\mathbf{I}$ then
    \begin{equation}
    \label{eq:identity-case}
        \E_{U_j}[(U_j^\dagger M_{\alpha_j}U_j)\otimes(U_j^\dagger M_{\beta_j}U_j)]=\mathbf{I}\otimes\mathbf{I}\,.
    \end{equation}
    Now suppose at least one of $M_{\alpha_j},M_{\beta_j}\neq \mathbf{I}$.
    Then $\tr[M_{\alpha_j}\otimes M_{\beta_j}]=0$ because non-identity elements of $\mathrm{GM}(\go)$ are traceless.
    Letting $\mathbf{F} := \sum_{j,k=1}^\go\ketbra{jk}{kj}$ denote a generalized \textsf{SWAP} operator, we have by the previous observation and properties of unitary 2-designs (see \emph{e.g.}, \cite[\S II.A]{GAE}) that
    \begin{align}
        \E_{U_j}\big[(U_j^\dagger M_{\alpha_j}U_j)\otimes (U_j^\dagger M_{\beta_j}U_j)\big] &= \frac{\tr[(\mathbf{I}+\mathbf{F})\, M_{\alpha_j}\!\otimes M_{\beta_j}]}{\go^2+\go}\left(\frac{\mathbf{I}+\mathbf{F}}{2}\right) + \frac{\tr[(\mathbf{I}-\mathbf{F})\,M_{\alpha_j}\!\otimes  M_{\beta_j}]}{\go^2-\go}\left(\frac{\mathbf{I}-\mathbf{F}}{2}\right)\nonumber\\
        &= \tr[\mathbf{F}\cdot M_{\alpha_j}\!\otimes M_{\beta_j}]\cdot\frac{1}{\go^2-1}\left(\mathbf{F}-\frac{1}{\go}\mathbf{I}\right)\,.
        \label{eq:2-design}
    \end{align}
    
    Let $S_\alpha = \{j:\alpha_j\neq \mathbf{I}\}$.
    Then combining Eqs. \eqref{eq:identity-case}, \eqref{eq:2-design} with the fact $\tr[\mathbf{F}\cdot M_{\alpha_j}\otimes M_{\beta_j}] = \tr[M_{\alpha_j}M_{\beta_j}] = \go\delta_{{\alpha_j}{\beta_j}}$ we have
    \begin{align}
        \label{eq:star-delta-and-flip}
        (*)&= \delta_{\alpha\beta}\textstyle\left(\frac{\go}{\go^2-1}\right)^{|\alpha|}\E_\rho\tr\big[(\mathbf{F}-\tfrac{1}{\go}\mathbf{I})^{\otimes|\alpha|}\rho_{S_\alpha}\!\otimes\rho_{S_\alpha}\big]\,.
    \end{align}
    Now we claim for any $\rho$ on $m$ qudits,
    \begin{equation}
    \label{eq:F-bound}
        \tr\big[(\mathbf{F}-\tfrac{1}{\go}\mathbf{I})^{\otimes m}\rho\otimes\rho\big]\leq 1\,.
    \end{equation}
    To see this, notice $\mathbf{F}=\frac{1}{\go}\sum_{\alpha}M_\alpha\otimes M_\alpha$, a sum of all GM basis elements doubled-up.
    Hence
    \begin{align*}
        \tr\big[(\mathbf{F}-\tfrac{1}{\go}\mathbf{I})^{\otimes m}\rho\otimes\rho\big] &=\tr\left[{\textstyle \bigotimes_{j=1}^m}\left(\tfrac{1}{\go}{\textstyle\sum_{\beta_j\neq \mathrm{Id}}}\,M_{\beta_j}\otimes M_{\beta_j} \right)\cdot\rho\otimes\rho\right] \\
        &= \frac{1}{\go^n}\sum_{\beta, |\beta|=m}\tr[M_\beta\otimes M_\beta\cdot\rho\otimes\rho]\\
        &\leq \frac{1}{\go^n}\sum_{\text{all}\,\beta}\tr[M_\beta \rho]^2\\
        &= \tr[\rho^2]\;\;\leq\;\; 1\,.
    \end{align*}
    Combining Eqs. \eqref{eq:star-delta-and-flip} and \eqref{eq:F-bound} we in summary have the bound
    \[(*)\leq \delta_{\alpha\beta}\left(\frac{\go}{\go^2-1}\right)^{|\alpha|}\,.\]
    Returning to Eq. \eqref{eq:difference-decomp} we conclude
    \begin{align*}
        \E_{\rho\sim\mu}|\tr[\cA\rho]|^2 \;= \sum_{\alpha} \left(\frac{\go}{\go^2-1}\right)^{-|\alpha|}|\widehat{\cA}(\alpha)|^2\,. & \qedhere
    \end{align*}
\end{proof}

\begin{definition}
    \label{defn:truncation}
    Let $\cA$ be an operator with Gell-Mann decomposition $\cA =\sum_\alpha \widehat{\cA}_\alpha  M_\alpha$.
    Then for $d\in[n]$ define its \emph{degree-$d$ truncation} to be $\cA^{\leq d} = \sum_{\alpha, |\alpha|\leq d}\widehat{\cA}_\alpha M_\alpha$.
\end{definition}
\begin{remark}
The choice of the GM decomposition here is essentially without loss of generality:
consider any basis $B$ for $M_\go(\C)$ containing the identity and define $\cA^{\leq d}_B$ analogously to Definition \ref{defn:truncation}, keeping the definition of degree analogous to that for the GM basis too.
Then we have $\cA^{\leq d}_B=\cA^{\leq d}$, as can be seen easily by expanding one basis in the other.
\end{remark}

\begin{corollary}
\label{cor:low-degree-approx}
$\E_{\rho\sim \mu}|\tr[\cA\rho]-\tr[\cA^{\leq d}\rho]|^2\leq \left(\frac{\go}{\go^2-1}\right)^{d}\|\cA\|_2^2$.
\end{corollary}
\begin{proof}
    We apply Theorem \ref{thm:L2DI-NS} to obtain
    \begin{align*}
    \E_{\rho\sim \mu}|\tr[\cA\rho]-\tr[\cA^{\leq d}\rho]|^2 &= \E_{\rho\sim \mu}|\tr[(\cA-\cA^{\leq d})\rho]|^2\\
    &\leq \sum_{\alpha, |\alpha|>d}\left(\tfrac{\go}{\go^2-1}\right)^{|\alpha|}|\widehat{\cA}(\alpha)|^2 \leq \left(\tfrac{\go}{\go^2-1}\right)^d\|\cA\|_2^2\,.\qedhere
    \end{align*}
\end{proof}

\begin{theorem}
    Let $\cA$ be any observable on $\mathcal{H}_\go^{\otimes n}$, of any degree, with $\|\cA\|_\mathrm{op}\leq 1$.
    Fix an error threshold $\epsilon >0$ and a failure probability $\delta > 0$ and put $t=\log_{\go^2-1}(4/\epsilon)$.
    Then there is a set $S$ of product states such that with a number
    \[s=\mathcal{O}\Big(\go^{3/2}\log\big(\tfrac{n}{\delta}\big)e^{c\cdot\log^2(\frac{1}{\eps})}\|\cA^{\leq t}\|^{2t}_\mathrm{op}\Big)\]
    of samples $(\rho, \tr[\cA \rho])$, $\rho\sim\mathcal{U}(S)$, an approximate operator $\widetilde{\cA}$ may be formed in time $\mathrm{poly}(n)$ with confidence $1-\delta$ such that
    \[\E_{\rho\sim\mu}|\tr[\widetilde{\cA}\rho]-\tr[\cA\rho]|^2\leq \eps\]
    for any L2DI distribution $\mu$.
\end{theorem}
\begin{proof}
    Choose the truncation degree $d=\log_{\go^2-1}(4/\eps)$.
    Then the triangle inequality and Corollary \ref{cor:low-degree-approx} give
    \begin{align*}
        \E_{\rho\sim \mu}|\tr[\widetilde{\cA}^{\leq d}\rho]-\tr[\cA\rho]|^2 & \leq 2\E_{\rho\sim \mu}|\tr[{\cA}^{\leq d}\rho]-\tr[\cA\rho]|^2+2\E_{\rho\sim \mu}|\tr[\widetilde{\cA}^{\leq d}\rho]-\tr[\cA^{\leq d}\rho]|^2\\
        &\leq 2\left(\tfrac{1}{\go^2-1}\right)^d+2\|\widetilde{\cA}^{\leq d} - \cA^{\leq d}\|_2^2\\
        &\leq \eps/2 + 2\|\widetilde{\cA}^{\leq d} - \cA^{\leq d}\|_2^2\,.
    \end{align*}
    So we need to choose a number of samples such that with confidence $1-\delta$, the low-degree qudit learning algorithm (Theorem \ref{thm:qudit-ld-learning})  yields a $\cA^{\leq d}$ such that $\|\widetilde{\cA}^{\leq d} - \cA^{\leq d}\|_2^2\leq\eps/4$.
    This requires no more than
    \[C\go^{3/2}\log\left(\frac{2en}{\delta}\right)e^{C'\log^2(4/\eps)}\|\cA^{\leq t}\|^{2t}_\mathrm{op}\]
    samples, where $t=\log_{\go^2-1}(4/\epsilon)$ and $C,C'$ are constants $> 1$.
\end{proof}

This learning theorem may be of interest even in the context of qubits.
In particular, for a small divisor $k$ of $n$, a system of $n$ qubits may be interpreted as $n/k$-many $2^k$-level qudits, and there may be interesting distributions over states in this system which are only L2DI when viewed as ``virtual qudits'' in this way. 

\section{Conclusions}

Our efforts to extend recent low-degree learning results to new spaces have led to discoveries in approximation theory, namely a dimension-free Remez-type inequality on the polytorus (Theorem \ref{thm:remez}).
It would be nice to find more applications of this inequality.

In the quantum setting, obtaining qudit BH inequalities required understanding the relationships among eigenspaces of basis elements in the Gell-Mann and Heisenberg--Weyl bases.
This is mostly complete, though it remains open what can be said for the HW basis when $K$ is composite.
And regarding the constants in the quantum BH inequalities, it is interesting to consider whether the exponential dependence of $\BH^{\leq d}_{M_2},\BH^{\leq d}_{\mathrm{GM}(\go)},$ and $\BH^{\leq d}_{\mathrm{HW}(\go)}$ on $d$ is necessary.
(Recall that in the BH inequalities for the polytorus and hypercube, the best known constant is subexponential in $d$).



\appendix


\begin{thebibliography}{10}

\bibitem{AEHK}
Ali Asadian, Paul Erker, Marcus Huber, and Claude Kl\"ockl.
\newblock Heisenberg-{W}eyl observables: Bloch vectors in phase space.
\newblock {\em Phys. Rev. A}, 94:010301, Jul 2016.
\newblock URL: \url{https://link.aps.org/doi/10.1103/PhysRevA.94.010301}, \href
  {https://doi.org/10.1103/PhysRevA.94.010301}
  {\path{doi:10.1103/PhysRevA.94.010301}}.

\bibitem{BPS}
Fr\'{e}d\'{e}ric Bayart, Daniel Pellegrino, and Juan~B. Seoane-Sep\'{u}lveda.
\newblock The {B}ohr radius of the $n$-dimensional polydisk is equivalent to
  $(\log n)/n$.
\newblock {\em Advances in Mathematics}, 264:726--746, 2014.
\newblock URL:
  \url{https://www.sciencedirect.com/science/article/pii/S000187081400262X},
  \href {https://doi.org/https://doi.org/10.1016/j.aim.2014.07.029}
  {\path{doi:https://doi.org/10.1016/j.aim.2014.07.029}}.

\bibitem{BH}
H.~F. Bohnenblust and Einar Hille.
\newblock On the absolute convergence of dirichlet series.
\newblock {\em Annals of Mathematics}, 32(3):600--622, 1931.
\newblock URL: \url{http://www.jstor.org/stable/1968255}.

\bibitem{BGKKL}
Jean Bourgain, Jeff Kahn, Gil Kalai, Yitzhak Katznelson, and Nathan Linial.
\newblock The influence of variables in product spaces.
\newblock {\em Israel Journal of Mathematics}, 77(1-2):55--64, February 1992.
\newblock \href {https://doi.org/10.1007/bf02808010}
  {\path{doi:10.1007/bf02808010}}.

\bibitem{dankert}
Christoph Dankert.
\newblock Efficient simulation of random quantum states and operators, 2005.
\newblock URL: \url{https://arxiv.org/abs/quant-ph/0512217}, \href
  {https://doi.org/10.48550/ARXIV.QUANT-PH/0512217}
  {\path{doi:10.48550/ARXIV.QUANT-PH/0512217}}.

\bibitem{DFOOS}
Andreas Defant, Leonhard Frerick, Joaquim Ortega-Cerd{\`{a}}, Myriam
  Ouna\"{\i}es, and Kristian Seip.
\newblock The {B}ohnenblust-{H}ille inequality for homogeneous polynomials is
  hypercontractive.
\newblock {\em Annals of Mathematics}, 174(1):485--497, July 2011.
\newblock \href {https://doi.org/10.4007/annals.2011.174.1.13}
  {\path{doi:10.4007/annals.2011.174.1.13}}.

\bibitem{DGMS}
Andreas Defant, Domingo Garc\'{\i}a, Manuel Maestre, and Pablo Sevilla-Peris.
\newblock {\em Dirichlet Series and Holomorphic Functions in High Dimensions}.
\newblock New Mathematical Monographs. Cambridge University Press, 2019.
\newblock \href {https://doi.org/10.1017/9781108691611}
  {\path{doi:10.1017/9781108691611}}.

\bibitem{DMP}
Andreas Defant, Mieczys{\l}aw Masty{\l}o, and Antonio P{\'{e}}rez.
\newblock On the {F}ourier spectrum of functions on boolean cubes.
\newblock {\em Mathematische Annalen}, 374(1-2):653--680, September 2018.
\newblock \href {https://doi.org/10.1007/s00208-018-1756-y}
  {\path{doi:10.1007/s00208-018-1756-y}}.

\bibitem{DS}
Andreas Defant and Pablo Sevilla-Peris.
\newblock {The Bohnenblust-Hille cycle of ideas from a modern point of view}.
\newblock {\em Functiones et Approximatio Commentarii Mathematici}, 50(1):55 --
  127, 2014.
\newblock \href {https://doi.org/10.7169/facm/2014.50.1.2}
  {\path{doi:10.7169/facm/2014.50.1.2}}.

\bibitem{EI22}
Alexandros Eskenazis and Paata Ivanisvili.
\newblock Learning low-degree functions from a logarithmic number of random
  queries.
\newblock In {\em Proceedings of the 54th Annual ACM SIGACT Symposium on Theory
  of Computing}, STOC 2022, page 203–207, New York, NY, USA, 2022.
  Association for Computing Machinery.
\newblock \href {https://doi.org/10.1145/3519935.3519981}
  {\path{doi:10.1145/3519935.3519981}}.

\bibitem{PhysRevLett.129.160501}
Daniel Gonz\'alez-Cuadra, Torsten~V. Zache, Jose Carrasco, Barbara Kraus, and
  Peter Zoller.
\newblock Hardware efficient quantum simulation of non-abelian gauge theories
  with qudits on {R}ydberg platforms.
\newblock {\em Phys. Rev. Lett.}, 129:160501, Oct 2022.
\newblock URL: \url{https://link.aps.org/doi/10.1103/PhysRevLett.129.160501},
  \href {https://doi.org/10.1103/PhysRevLett.129.160501}
  {\path{doi:10.1103/PhysRevLett.129.160501}}.

\bibitem{GAE}
D.~Gross, K.~Audenaert, and J.~Eisert.
\newblock Evenly distributed unitaries: On the structure of unitary designs.
\newblock {\em Journal of Mathematical Physics}, 48(5):052104, 2007.
\newblock \href {http://arxiv.org/abs/https://doi.org/10.1063/1.2716992}
  {\path{arXiv:https://doi.org/10.1063/1.2716992}}, \href
  {https://doi.org/10.1063/1.2716992} {\path{doi:10.1063/1.2716992}}.

\bibitem{CHP}
Hsin-Yuan Huang, Sitan Chen, and John Preskill.
\newblock Learning to predict arbitrary quantum processes.
\newblock {\em PRX Quantum}, 4(4):040337, 2023.

\bibitem{Huang2020}
Hsin-Yuan Huang, Richard Kueng, and John Preskill.
\newblock Predicting many properties of a quantum system from very few
  measurements.
\newblock {\em Nature Physics}, 16(10):1050--1057, June 2020.
\newblock \href {https://doi.org/10.1038/s41567-020-0932-7}
  {\path{doi:10.1038/s41567-020-0932-7}}.

\bibitem{fouriergrowth}
Siddharth Iyer, Anup Rao, Victor Reis, Thomas Rothvoss, and Amir Yehudayoff.
\newblock Tight bounds on the {F}ourier growth of bounded functions on the
  hypercube.
\newblock {\em Electron. Colloquium Comput. Complex.}, {TR21-102}, 2021.
\newblock URL: \url{https://eccc.weizmann.ac.il/report/2021/102}, \href
  {http://arxiv.org/abs/TR21-102} {\path{arXiv:TR21-102}}.

\bibitem{quant-sim}
Doga~Murat Kurkcuoglu, M.~Sohaib Alam, Joshua~Adam Job, Andy C.~Y. Li,
  Alexandru Macridin, Gabriel~N. Perdue, and Stephen Providence.
\newblock Quantum simulation of $\phi^4$ theories in qudit systems.
\newblock 2021.
\newblock URL: \url{https://arxiv.org/abs/2108.13357}, \href
  {https://doi.org/10.48550/ARXIV.2108.13357}
  {\path{doi:10.48550/ARXIV.2108.13357}}.

\bibitem{LMN}
Nathan Linial, Yishay Mansour, and Noam Nisan.
\newblock Constant depth circuits, fourier transform, and learnability.
\newblock {\em J. ACM}, 40(3):607–620, jul 1993.
\newblock \href {https://doi.org/10.1145/174130.174138}
  {\path{doi:10.1145/174130.174138}}.

\bibitem{Odonnell}
Ryan O'Donnell.
\newblock {\em Analysis of Boolean Functions}.
\newblock Cambridge University Press, 2014.
\newblock \href {https://doi.org/10.1017/CBO9781139814782}
  {\path{doi:10.1017/CBO9781139814782}}.

\bibitem{RWZ22}
Cambyse Rouz\'e, Melchior Wirth, and Haonan Zhang.
\newblock {Quantum {T}alagrand, {KKL} and {F}riedgut's theorems and the
  learnability of quantum {B}oolean functions}.
\newblock 9 2022.
\newblock \href {http://arxiv.org/abs/2209.07279} {\path{arXiv:2209.07279}}.

\bibitem{Scott_2006}
A~J Scott.
\newblock Tight informationally complete quantum measurements.
\newblock {\em Journal of Physics A: Mathematical and General}, 39(43):13507,
  oct 2006.
\newblock URL: \url{https://dx.doi.org/10.1088/0305-4470/39/43/009}, \href
  {https://doi.org/10.1088/0305-4470/39/43/009}
  {\path{doi:10.1088/0305-4470/39/43/009}}.

\bibitem{SVZremez}
Joseph Slote, Alexander Volberg, and Haonan Zhang.
\newblock A dimension-free remez-type inequality on the polytorus, 2023.
\newblock \href {http://arxiv.org/abs/arXiv:2305.10828}
  {\path{arXiv:arXiv:2305.10828}}.

\bibitem{VZ22}
Alexander Volberg and Haonan Zhang.
\newblock Noncommutative {B}ohnenblust--{H}ille inequalities.
\newblock {\em Mathematische Annalen}, page 1–20, 2023.
\newblock \href {https://doi.org/10.1007/s00208-023-02680-0}
  {\path{doi:10.1007/s00208-023-02680-0}}.

\bibitem{quditsurvey}
Yuchen Wang, Zixuan Hu, Barry~C. Sanders, and Sabre Kais.
\newblock Qudits and high-dimensional quantum computing.
\newblock {\em Frontiers in Physics}, 8, 2020.
\newblock URL:
  \url{https://www.frontiersin.org/articles/10.3389/fphy.2020.589504}, \href
  {https://doi.org/10.3389/fphy.2020.589504}
  {\path{doi:10.3389/fphy.2020.589504}}.

\bibitem{Yomdin2011}
Y.~Yomdin.
\newblock Remez-type inequality for discrete sets.
\newblock {\em Israel Journal of Mathematics}, 186(1):45–60, November 2011.
\newblock URL: \url{http://dx.doi.org/10.1007/s11856-011-0131-4}, \href
  {https://doi.org/10.1007/s11856-011-0131-4}
  {\path{doi:10.1007/s11856-011-0131-4}}.

\end{thebibliography}
\end{document}